\input ./style/arxiv-general.cfg
\documentclass[aop,MSNbibl,secthm,seceqn,dvips]{arximspdf}
\makeatletter
   \@ifpackageloaded{graphicx}{}{\usepackage{graphicx}}
\makeatother
\usepackage{mathrsfs}
\usepackage{mathbh}

\doi{10.1214/14-AOP992}% Updated by VTEXPTS2LaTeX.exe, 20.01.2015 09:03
\volume{44}
\issue{2}
\pubyear{2016}
\firstpage{924}
\lastpage{954}
\docsubty{FLA}

\makeatletter
\newcommand{\ds}{\displaystyle}

\newcommand{\rrvert}{\vert}

\newcommand{\llvert}{\vert}
\newcommand{\eqref}[1]{(\ref{#1})}
\newproclaim{defn}{Definition}[section]
\newtheorem{prop}{Proposition}[section]
\newtheorem{cor}{Corollary}[section]
\newproclaim{rmk}{Remark}[section]
\newtheorem{lma}{Lemma}[section]
\newproclaim{exm}{Example}[section]
\newproclaim{con}{Conjecture}
\newcommand{\ud}{d}
\newcommand{\LL}{\mathbf{L}}
\newcommand{\PP}{\mathscr{P}}
\newcommand{\K}{\mathbf{Ker}}
\newcommand{\Id}{\mathbf{Id}}
\newcommand{\N}{\mathbb{N}}
\newcommand{\R}{\mathbb{R}}
\newcommand{\E}{\mathbb{E}}

\renewcommand{\P}{\mathbb{P}}
\makeatother

\begin{document}
\begin{frontmatter}

\title{Generalization of the Nualart--Peccati criterion}
\runtitle{Generalized Nualart--Peccati criterion}

\begin{aug}
% Corresponding author: Ehsan Azmoodeh - ehsan.azmoodeh@uni.lu% Updated by VTEXPTS2LaTeX.exe, 23.01.2015 15:00
%by VTEXPTS2LaTeX.exe, 20.01.2015 09:03
\author[A]{\fnms{Ehsan}~\snm{Azmoodeh}\corref{}\ead[label=e1]{ehsan.azmoodeh@uni.lu}\thanksref{m1}\thanksref{T1}},
\author[B]{\fnms{Dominique}~\snm{Malicet}\ead[label=e2]{Malicet.dominique@crans.org}\thanksref{m2}},
\author[C]{\fnms{Guillaume}~\snm{Mijoule}\ead[label=e3]{guillaume.mijoule@m4x.org}\thanksref{m3}}
\and
\author[D]{\fnms{Guillaume}~\snm{Poly}\ead[label=e4]{guillaume.poly@univ-rennes1.fr}\thanksref{m4}\thanksref{T2}}
\runauthor{Azmoodeh, Malicet, Mijoule and Poly}
\affiliation{University of Luxembourg\thanksmark{m1}, PUC-Rio\thanksmark{m2},  Universit\'{e} Paris-Sud 11\thanksmark{m3} and University Rennes 1\thanksmark{m4}}
%\dedicated{}
\address[A]{E. Azmoodeh\\
Mathematices Research Unit\\
Universit\'e du Luxembourg\\
Luxembourg City, 1359\\
Luxembourg\\
\printead{e1}}
\address[B]{D. Malicet\\
Departamento de Matem\'atica\\
PUC-Rio\\
Rio de Janeiro, 22451-900\\
Brazil\\
\printead{e2}\hspace*{23pt}}
\address[C]{G. Mijoule\\
D\'{e}partment de Math\'{e}matiques\\
Facult\'{e} des Sciences d'Orsay\\
Universit\'{e} Paris-Sud 11\\
F-91405 Orsay Cedex\\
France\\
\printead{e3}}
\address[D]{G. Poly\\
Institut de Recherche Math\'{e}matiques \\
Universit\'e de Rennes 1\\
35042, Rennes C\'edex\\
France\\
\printead{e4}}
\end{aug}
\thankstext{T1}{Supported by research project
F1R-MTH-PUL-12PAMP.}
\thankstext{T2}{Supported by AFR Grant 4897114.}

% HISTORY:
%
\received{\smonth{11} \syear{2013}}% Updated by VTEXPTS2LaTeX.exe,
%20.01.2015 09:03
%
\revised{\smonth{12} \syear{2014}}% Updated by VTEXPTS2LaTeX.exe,
%20.01.2015 09:03

% ABSTRACT
\begin{abstract}
The celebrated Nualart--Peccati criterion
[\textit{Ann. Probab.} \textbf{33} (2005) 177--193]
ensures the
convergence in distribution toward a standard Gaussian random variable
$N$ of a given sequence $\{X_n\}_{n\ge1}$ of multiple\vspace*{1pt} Wiener--It\^{o}
integrals of fixed order, if $\E[X_n^2]\to1$ and $\E[X_n^4]\to\E
[N^4]=3$. Since its appearance in 2005, the natural question of
ascertaining which
other moments can replace the fourth moment in the above criterion has
remained entirely open. Based on the technique recently
introduced in
[\textit{J. Funct. Anal.} \textbf{266} (2014) 2341--2359], we settle this problem and establish that
the convergence of \textit{any even moment}, greater than four, to the
corresponding moment of the standard Gaussian distribution, guarantees
the central convergence. As a by-product, we provide many new moment
inequalities for multiple Wiener--It\^{o} integrals. For instance, if
$X$ is a normalized multiple Wiener--It\^o integral of order greater than
one,
\[
\forall k \ge2, \qquad \E\bigl[X^{2k}\bigr] > \E\bigl[N^{2k}
\bigr]=(2k-1)!!.
\]
\end{abstract}

% KEYWORDS
\end{frontmatter}

%s1 #&#
\section{Introduction and summary of the main results}\label{sec1}

Let $\{B_t\}_{t\ge0}$ be a standard Brownian motion and $p$ be an
integer greater than $1$. For any deterministic and symmetric function
$f\in L^2(\R_+^p,\lambda_p)$ ($\lambda_p$ stands for $p$-dimensional
Lebesgue measure), let $I_p(f)$ be the $p{\mathrm{th}}$ multiple
Wiener--It\^{o} integral of $f$ with respect to $\{B_t\}_{t\ge0}$ (see
\cite{Nu-book} for a precise definition). The vector space spanned by all
the multiple integrals of order $p$ is called the $p{\mathrm{th}}$
Wiener chaos. A fundamental result of stochastic calculus, customarily
called the \textit{Wiener--It\^{o} decomposition}, asserts that any
square integrable functional of $\{B_t\}_{t\ge0}$ can be uniquely
expanded as an orthogonal sum of multiple Wiener--It\^{o} integrals. As
such, the study of the properties of multiple Wiener--It\^{o} integrals
becomes a central topic of research in modern stochastic analysis and a
great part of the so-called \textit{Malliavin calculus} (see, e.g.,
\cite{Nu-book,n-p}) relies on it.

The following result, nowadays known as the \textit{fourth moment
theorem}, yields an effective criterion of central convergence for a
given sequence of multiple Wiener--It\^{o} integrals of a fixed order.
%
%th1.1 #&#
\begin{thm}[(Nualart--Peccati \cite{n-p-05})]\label{peccati}
Let $p \ge2$ and $f_n$ be a sequence of symmetric elements of $L^2(\R
_+^p,\lambda_p)$. Assume $X_n=I_p(f_n)$ verifies $ \E [X_n^2
]\to1$. Then, as $n \to\infty$,
\[
\label{CLTFM}
X_n \stackrel{\mathit{law}}{\rightarrow}N \sim
\mathcal{N}(0,1) \quad \mbox{if and only if}\quad \E\bigl[X_n^4
\bigr]\to\E\bigl[N^4\bigr]=3.
\]
\end{thm}

The main goal of this article is to show that the above theorem is a
particular case of a more general phenomenon. This is the content of
the next theorem.

%th1.2 #&#
\begin{thm}\label{CLTFMquibourre}
Under the assumptions of Theorem~\ref{peccati}, for any integer $k\ge
2$, as $n \to\infty$,
\[
\label{CLTFM-ultimate}
X_n \stackrel{\mathit{law}} {\rightarrow}N \sim
\mathcal{N}(0,1)\quad \mbox{if and only if}\quad \E\bigl[X_n^{2k}
\bigr]\to\E\bigl[N^{2k}\bigr]=(2k-1)!!,
\]
where the double factorial is defined by $(2k-1)!!=\prod_{i=1}^k (2i-1)$.
\end{thm}

The discovery of the fourth moment theorem by  Nualart and  Peccati
(see \cite{n-p-05}) is arguably a major breakthrough in the field of
Gaussian approximation in the Wiener space. It resulted in a
drastic simplification of the so-called \textit{method of moments}
(consisting in checking the convergence of \textit{all} the moments)
which was so far the only alternative to establish a central limit
theorem. We refer to Breuer and Major \cite{BM}, Chambers and Slud \cite
{CS}, Surgailis \cite{SU} and \cite{NS11,GW} for a nonexhaustive
exposition of some results provided by the method of moments. The first
proof of Theorem~\ref{peccati} relies on tools from stochastic analysis (namely the
Dambis, Dubins and Schwartz's theorem; see, e.g.,
\cite{RY}, Chapter V). Later on, in the seminal paper \cite{n-o},  Nualart and
Ortiz-Latorre discovered a fundamental link between the central
convergence of a sequence of elements of a fixed Wiener chaos and the
convergence to a constant of the norms of their Malliavin derivatives.
The role played
by the Malliavin derivative in the fourth moment theorem was later on
confirmed in the landmark paper \cite{n-p-2}. There,  Nourdin and
Peccati combined Stein's method and Malliavin calculus to provide a new
proof of Theorem~\ref{peccati} culminating in sharp estimates for
various distances. As an illustrative example of such estimates, they
could prove the following quantitative version of Theorem~\ref{peccati}.

%th1.3 #&#
\begin{thm}[(Nourdin--Peccati \cite{n-p-2})]\label{quantificationtheo}
Let $p\ge2$. Assume\vspace*{1pt} that $X=I_p(f)$ where $f$ is a symmetric element
of $L^2(\R_+^p,\lambda_p)$ such that $\E[X^2]=1$. Then,
%
%e1.1 #&#
\begin{equation}
\label{quantification}
d_{\mathrm{TV}} \bigl(X,\mathcal{N}(0,1) \bigr) \le
\tfrac{2}{\sqrt{3}} \sqrt{\E \bigl[X^4 \bigr]-3},
\end{equation}
where $d_{\mathrm{TV}}$ stands for the total variation distance.
\end{thm}

This innovative approach, combining Malliavin calculus and Stein's\break 
method, gave a new impetus to the well studied field of normal
approximation within Gaussian spaces. Indeed, it resulted in
spectacular improvements of many classical results previously obtained
by the
method of moments. For an exposition of this fertile line of research,
one can consult the book \cite{n-p}, the surveys \cite{p14survey,C14}
and the following frequently updated webpage which aims at
referencing all the articles dealing with the so-called \textit{Malliavin--Stein method}:
\surl{https://sites.google.com/site/malliavinstein/home}.
Finally, we mention that Theorem~\ref{quantificationtheo} has been
generalized in various directions such as the optimality of the rate of
convergence \cite{Optimal}, multivariate settings
\cite{Tudor,Entropy}, free probability settings \cite{K12} and general
homogeneous sums \cite{npsp14}.

Unfortunately, if the quantitative aspects of the latter approach are
now quite well understood, the heavily combinatorial nature of the
proofs remains a major stumbling block in tackling the following
central questions:
\begin{longlist}[(C)]
\item[(A)] What are the target distributions for which a moment
criterion similar to~(\ref{CLTFM}) is available?
\item[(B)] What are the moment conditions ensuring the central convergence?
\item[(C)] What are the special properties of Wiener chaos playing a
role in the fourth moment phenomenon?
\end{longlist}
Indeed, most of the aforementioned proofs of the fourth moment theorem
make crucial use of the product formula for multiple Wiener--It\^{o}
integrals together with some properties of the underlying Fock space.
Such an approach becomes already inextricable when one tries to make
explicit the Wiener--It\^{o} decomposition of the $6$th power
of a multiple integral. As such, writing explicitly the Wiener--It\^{o}
decompositions of the successive powers of a given multiple
Wiener--It\^{o} integral, which is the core of the previous strategies,
seems totally hopeless for our purpose. Inspired by the remarkable intuition
that the fourth moment phenomenon could also be explained by the
spectral properties of Markov operators, Ledoux produced a new proof
of Theorem~\ref{peccati} (see \cite{le}). His approach, exclusively
based on the study of some differential operators such as the
Ornstein--Uhlenbeck generator and the iterated gradients, and avoids
completely the product formula for multiple integrals.\vadjust{\goodbreak} Unfortunately,
due to
an inappropriate definition of the chaos of a Markov operator, this attempt
became rather involved and could not produce any of the expected
generalizations of the fourth moment criterion. Later on, in the same
spirit as \cite{le} (i.e., exploiting spectral theory of Markov operators
and Gamma-calculus), the authors of \cite{a-c-p} could produce a very
simple and fully transparent proof of the fourth moment theorem,
henceforth bringing both, a complete answer to question (C), and some
generalizations of the criterion for other Markov operators than
Ornstein--Uhlenbeck. Roughly speaking, the technique used in \cite{a-c-p} consists in exploiting the stability of chaoses under the
product operation to provide some suitable spectral
inequalities, which, after elementary computations, become moment
inequalities. In particular, the latter approach does not need any of
the combinatorics computations required by the product formula. The
present article fully generalizes this idea and builds a complete
methodology enabling to provide a wide range of inequalities for
polynomial functionals of a Gaussian field, hence for Wiener chaoses as well.
In particular, it leads to a partial answer to question (B). Combining
the formalism and the ideas of \cite{a-c-p} together with some fine
deterministic properties of Hermite polynomials, we could prove
Theorem~\ref{CLTFMquibourre}, which is our main achievement.
Interestingly, we could also prove the following quantitative version
which extends the celebrated estimate (\ref{quantification}) to all
even moments. Indeed, taking $k=2$ in the theorem below gives back the
bound in\break (\ref{quantification}).
%
%th1.4 #&#
\begin{thm}\label{quantificationdeouf}
Under the assumptions of Theorem~\ref{quantificationtheo}, for all
$k\ge2$, we have the following general quantitative bound:
%
%e1.2 #&#
\begin{equation}
\label{superquantification}
d_{\mathrm{TV}} \bigl(X,\mathcal{N}(0,1) \bigr)\le C_k
\sqrt{\frac{\E
[X^{2k} ]}{(2k-1)!!}-1},
\end{equation}
where the constant $C_k=\frac{4}{\sqrt{ 2k(k-1) \int_0^1 (({1+t^2})/{2})^{k-2} \,dt}}$.
\end{thm}

For the sake of clarity, we stated so far our main results in the more
familiar context of the Wiener space. Nevertheless, throughout the
whole article, instead of the Wiener--It\^{o} multiple integrals, we
shall consider a more general concept of eigenfunctions of a diffusive
Markov operator. We refer the reader to Section~\ref{assumptions}
for a precise exposition of our assumptions. We also stress
that this gain of generality enables us to give central limit criteria
in situations far beyond the scope of the usual criteria holding in the
Wiener chaoses.\vadjust{\eject}

%s2 #&#
\section{The setup}\label{setup}
%s2.1 #&#
\subsection{The general setup and assumptions \textup{(a)}--\textup{(b)}--\textup{(c)}}\label{assumptions}
One possible way to study the properties of the elements of the Wiener
space can be through the exploration of the spectral properties of the
so-called \textit{Ornstein--Uhlenbeck operator}. In order to
situate more precisely our purpose, we will restate below its main
properties.

The Wiener space is the space $L^2(E,\mu)$ where $E=\R^\N$ and $\mu$ is
the standard Gaussian measure on $\R^\N$. The Ornstein--Uhlenbeck
operator is the unbounded, symmetric, negative operator $\LL$ acting
on some dense domain of $L^2(E,\mu)$ and defined (on the set of smooth
enough cylindric functionals $\Phi$) by
\[
\LL [\Phi ]=\Delta\Phi- \vec{x}\cdot\vec{\nabla}\Phi
\]
(where $\Delta$ is the usual Laplacian and $\vec{\nabla}$ is the
gradient vector).
We shall denote its domain by $\mathcal{D}(\LL)$ and for a general $X\in
\mathcal{D}(\LL)$, $\LL [X ]$ is defined via standard closure
operations. The associated \textit{carr\'e-du-champ} operator is the
symmetric, positive, bilinear form defined by
\[
\Gamma[\Phi,\Psi]=\vec{\nabla}\Phi\cdot\vec{\nabla}\Psi.
\]

Below, we summarize the fundamental properties of the
Ornstein--Uhlenbeck operator $\LL$.
\begin{longlist}[(a)]
\item[(a)] \textit{Diffusion}: For any $\mathcal{C}^2_b$ function $\phi
\dvtx \R\to\R$, any $X \in\mathcal{D}(\LL)$, it holds that $\phi
(X)\in\mathcal{D}(\LL)$ and
%
%e2.1 #&#
\begin{equation}
\label{diff}
\LL \bigl[\phi(X) \bigr] = \phi'(X) \LL[X] +
\phi''(X) \Gamma[X,X].
\end{equation}
Note that, by taking $\phi=1 \in\mathcal{C}^2_b$, we get $\LL[1]=0$
which is the Markov property. Equivalently, $\Gamma$ is a derivation in
the sense that
\[
\Gamma \bigl[ \phi(X),X \bigr] = \phi'(X) \Gamma[X,X].
\]

\item[(b)] \textit{Spectral decomposition}: The operator $-\LL$ is
diagonalizable on $L^2(E,\mu)$ with $\mathbf{sp}(-\LL)=\N$, that is to say
\[
L^2(E,\mu)=\bigoplus_{i=0}^\infty
\K(\LL+i \Id).
\]
\item[(c)] \textit{Spectral stability}: For any pair of eigenfunctions
$(X,Y)$ of the operator $- \LL$ associated with eigenvalues $(p_1,p_2)$,
%
%e2.2 #&#
\begin{equation}
\label{fundamental-assumtionbis}
X Y\in\bigoplus_{i \le p_1+p_2 } \K ( \LL+ i \Id).
\end{equation}
\end{longlist}

We refer to \cite{BH} for a precise exposition as well as all the
domain and integrability assumptions. Actually, these three properties
are the only one we will use. Thus, we naturally define the following
class of structures for which our results will hold.

%de2.1 #&#
\begin{defn}
A (a)--(b)--(c) structure is a triplet $(E,\mu,\LL)$, with an associated
``carr\'e-du-champ'' operator $\Gamma$, where:
\begin{itemize}
\item $(E,\mu)$ is a probability space,
\item $\LL$ is a symmetric unbounded operator defined on some dense
domain of $L^2(E,\mu)$,
\item$\Gamma$ is
defined by
%
%e2.3 #&#
\begin{equation}
\label{Gamma}
2 \Gamma [X,Y ] = \LL [XY ] - X \LL [Y ] - Y \LL [X ],
\end{equation}
\end{itemize}
such that the aforementioned properties (a), (b) and (c) hold.
In this context, we will sometimes write $\Gamma[X]$ to denote $\Gamma
[X,X]$ and $\E$ for the integration against $\mu$.
\end{defn}

Property (a) is important regarding functional calculus. For instance,
we will use several times the following \textit{integration by parts
formula}: for any $X,Y$ in $\mathcal{D}(\LL)$ and
$\phi\in\mathcal{C}^2_b$:
%
%e2.4 #&#
\begin{equation}
\label{by-parts}
\E \bigl[\phi'(X)\Gamma [X,Y ] \bigr] = - \E \bigl[
\phi(X) \LL [Y ] \bigr] = -\E \bigl[Y \LL \bigl[\phi(X) \bigr] \bigr].
\end{equation}

Property (b) allows to use spectral theory. Actually, we stress that
our results extend under the weaker assumption that $\mathbf{sp}(-\LL
)\subset\R_+$ is simply discrete. However, we stick to the
assumption $\mathbf{sp}(-\LL)=\N$ since it encompasses the most common
cases (Wiener space and Laguerre space). The reader interested in
relaxing this spectral assumption can consult \cite{a-c-p} where the
spectrum is only assumed to be discrete.

Property (c) is our main assumption, which will allow us to obtain
fundamental spectral inequalities. A simple induction on (\ref{fundamental-assumtionbis}) shows that, for any $X\in\K(\LL+p\Id)$ and
any polynomial $P$ of degree $m$, we have
%
%e2.5 #&#
\begin{equation}
\label{fundamental-assumtion}
P(X) \in\bigoplus_{i \le mp } \K ( \LL+ i
\Id ).
\end{equation}
For further details on our setup, we refer to \cite{ba,b-g-l}. We also
refer to Section~\ref{examples-assumptions} for many other examples.

%re2.1 #&#
\begin{rmk}\label{Hyperconractivity}
We remark that under the assumptions (a)--(b)--(c), the eigenspaces are
hypercontractive (see \cite{ba} for sufficient conditions), that is,
for any integer $M$, we have that
%
%e2.6 #&#
\begin{equation}
\label{Hyper1}
\bigoplus_{i \le M} \K ( \LL+ i \Id )
\subseteq\bigcap_{p\ge1} L^{p}(E,\mu).
\end{equation}
Next, by using the open mapping theorem, we see that the embedding (\ref{Hyper1}) is continuous, that is, there exists a constant $C(M,k)$ such
that for any $X\in\bigoplus_{i \le M} \K ( \LL+ i \Id )$:
%
%e2.7 #&#
\begin{equation}
\label{Hyper2}
\E\bigl(X^{2k}\bigr)\leq C(M,k) \E\bigl(X^2
\bigr)^k.
\end{equation}
\end{rmk}

We close this subsection with an useful lemma which will be used
several times in the sequel. This lemma is proved in \cite{n-po-2}, Lemma~2.4,  in the Wiener structure but can be easily adapted to our
framework by taking into
account the Remark~\ref{Hyperconractivity}.

%le2.1 #&#
\begin{lma}\label{Hypercontract}
Let $\{ X_n\}_{n \ge1}$ be a sequence of random variables living in a
finite sum of eigenspaces\vspace*{1pt} $(X_n \in\bigoplus_{i \le M} \K ( \LL+
i \Id ),  \forall n\ge1)$ of a Markov generator $\LL$ such that
our assumptions \textup{(a)}, \textup{(b)}
and \textup{(c)} hold. Assume that the sequence $\{X_n\}_{n \ge1}$ converges in
distribution as $n$ tends to infinity. Then
\[
\sup_{n \ge1} \E\bigl(\vert X_n \vert^r
\bigr) < \infty \qquad  \forall r \ge1.
\]
\end{lma}

%s2.2 #&#
\subsection{Examples of structures fulfilling assumptions
\textup{(a)}--\textup{(b)}--\textup{(c)}}\label{examples-assumptions}

We refer to the article \cite{a-c-p} for a proof of the validity of the
assumptions (a)--(b)--(c) in the cases of the Wiener and Laguerre
structures. We now show how the validity of the assumptions (a)--(b)--(c)
is preserved by the elementary operations of tensorization and
superposition of structures. This simple fact will allow us to produce
many structures in which our results hold.

\subsubsection*{Tensorization} Let $(E_1,\mu_1,\LL_1)$ and
$(E_2,\mu_2,\LL_2)$ be two Markov triplets fulfilling assumptions
(a)--(b)--(c). On the product space $E_1\times E_2$ with measure
$\mu_1\otimes\mu_2$, we define the following operator $\LL_3$. For $\Psi
\dvtx E_1\times E_2\rightarrow\R$, we set $\Psi_x(y)=\Psi_y(x)=\Psi(x,y)$,
and we define
%
%e2.8 #&#
\begin{equation}
\label{tensor}
\LL_3[\Psi](x,y)=\LL_1[
\Psi_y](x)+\LL_2[\Psi_x](y).
\end{equation}
In (\ref{tensor}), $\LL_3$ is defined on the set of maps $\Psi$ such that:
\begin{longlist}[(2)]
\item[(1)] $\mu_2$-a.s., $\Psi_y\in\operatorname{Dom}(\LL_1)$ and $\mu
_1$-a.s., $\Psi_x\in\operatorname{Dom}(\LL_2)$,
\item[(2)]
\[
\int_{E_1\times E_2} \bigl(\LL_3[\Psi](x,y)
\bigr)^2\,\ud\mu_1 \,\ud\mu _2<\infty.
\]
\end{longlist}
We claim that the triplet $(E_1\times E_2,\mu_1\otimes\mu_2,\LL_3)$
verifies assumptions (a)--(b)--(c). First, it is well known that this
procedure preserves assumption (a); see, for instance, \cite{BH}.
Assumption (b) is also preserved by tensorization taking into account that
%
%e2.9 #&#
\begin{eqnarray}
 \K(\LL_3+k\Id)
& =&  \operatorname{Vect} \bigl\{\phi(x)
\psi(y) |\phi\in\K(\LL _1+k_1\Id), \nonumber\\[-8pt]
\label{chaos-tensor}
 \\[-8pt]
 \nonumber
&& \hspace*{29pt}\qquad\qquad
\psi\in\K(\LL_2+k_2\Id), k_1+k_2=k\bigr\}.
\end{eqnarray}
Finally, we check assumption (c) for $\LL_3$. Let $\Psi_1=\phi_1(x)\psi
_1(y)$ with $\phi_1\in\K(\LL+k_1\Id)$ and $\psi_1\in\K(\LL+k_2\Id)$,
and let $\Psi_2=\phi_2(x)\psi_2(y)$ where $\phi_2\in\K(\LL+k_3\Id)$
and $\psi_2\in\K(\LL+k_4\Id)$. By applying (c) to $\phi_1 \phi_2$ and
$\psi_1 \psi_2$ together with equation (\ref{chaos-tensor}), we infer that
\[
\phi_1(x)\phi_2(x)\psi_1(y)
\psi_2(y)\in\bigoplus_{i\leq
k_1+k_2+k_3+k_4}\K(
\LL_3+i\Id).
\]
Hence, using bilinearity, we see that assumption (c) also holds for
operator $\LL_3$.
\subsubsection*{Superposition}
As before, we are given a Markov
triplet $(E,\LL,\mu)$ satisfying assumptions (a)--(b)--(c). The
superposition procedure consists in adding an independent noise to
$(E,\LL,\mu)$. To do so, we consider a generic probability space
$(\Omega,\mathcal{F},\P)$, which will induce the noise on $(E,\LL,\mu
)$. We define on the set $E\times\Omega$ equipped with the product
probability measure $\mu\otimes\P$:
%
%e2.10 #&#
%e2.11 #&#
\begin{eqnarray}
\label{newdomainmixture}
\quad\qquad\operatorname{Dom}(\LL_\Omega) &=&  \biggl\{\Psi(x,\omega) \Big|
\Psi_\omega\in\mbox {Dom}(\LL), \int_\Omega
\E_\mu \bigl[ (\LL\Psi_\omega )^2 \bigr] \,\ud\P<
\infty \biggr\},
\\
\label{newgeneratormixture}
\LL_\Omega[\Psi](x,\omega)&:=& \LL[\Psi_\omega](x)\qquad
\forall \Psi \in\operatorname{Dom}(\LL_\Omega).
\end{eqnarray}
Preservation of assumption (a) is a well-known consequence of the
superposition procedure. We refer to \cite{BH} where
superposition/product/semidirect product of Markov triplets (i.e.,
Dirichlet forms)
are studied to provide ways of constructing Dirichlet forms. To check
assumption (b), we are given $\Psi(x,\omega)\in L^2(\mu\otimes\P)$. By
assumption (b) on the space $L^2(E,\mu)$, we get
%
%e2.12 #&#
\begin{equation}
\Psi_{\omega}(x)=\sum_{k=1}^\infty
f_{k,\omega}(x), \qquad f_{k,\omega}\in\K (\LL+k\Id).
\end{equation}
Besides,
\[
\int_\Omega\E_\mu \bigl[\Psi(x,
\omega)^2 \bigr]\,\ud\P=\sum_{k=1}^\infty
\int_\Omega\E_\mu \bigl[f_{k}(x,
\omega)^2 \bigr]\,\ud\P<\infty.
\]
This ensures that $\P$-a.s., $f_{k,\omega}\in\K(\LL+k\Id)$ and
that $f_{k}\in\operatorname{Dom}(\LL_\Omega)$. Finally, one can see that
%
%e2.13 #&#
\begin{equation}
\label{newchaosmixt}
\hspace*{9pt}\quad\K(\LL_\Omega+k\Id)= \bigl\{\Psi(x,\omega)\in\operatorname{Dom}(
\LL_\Omega)| \P\mbox{-a.s. } \Psi_\omega\in\K(\LL+k\Id) \bigr\}.
\end{equation}
We infer that $f_k\in\K(\LL_\Omega+k\Id)$ which achieves the proof of
(b). Strictly speaking, assumption (c) is not necessarily preserved
because we need integrability on the product of two eigenfunctions of
$\LL_\Omega$. This integrability, unlike in the tensorization procedure
is not automatically fulfilled in the superposition procedure.
Fortunately, under some slight additional assumption, (c) holds for $\LL
_\Omega$.
More precisely, we have for all $X(x,\omega)\in\K(\LL_\Omega
+k_1\Id)$ and $Y(x,\omega)\in\K(\LL_\Omega+k_2\Id)$ such that $X Y \in
L^2(\mu\otimes\P)$:
%
%e2.14 #&#
\begin{equation}
\label{c-mixt}
X Y\in\bigoplus_{i\leq k_1+k_2}\K(
\LL_\Omega+i\Id).
\end{equation}

%re2.2 #&#
\begin{rmk}
One can consult the reference \cite{b-g-l}, page 515,  to see that the two
aforementioned operations are a particular case of the so-called
wrapped product of symmetric diffusive operators.
\end{rmk}

%\begin{exm}
%Let us illustrate the efficiency of theses two procedures, to
%construct structures with properties (a)-(b)-(c). First we infer that
%both Wiener structure and Laguerre structure satisfy the required
%assumptions (a)-(b)-(c), we refer to \cite{a-c-p} for details. Let $
%\{T_i\}_{i\ge1}$ be \textbf{any} sequence of random variables (not
%necessarily independent nor identically distributed) defined on some
%probability space $(\Omega,\mathcal{F},\P)$. Let $\{X_i\}_{i\ge1}$ be
%an i.i.d. sequence of $\mathcal{N}(0,1)$-distributed random variables
%defined on the first chaos of the Wiener structure, say, $(E,\LL,
%\mu)$. Finally let
%$$F_n=\sum_{i_1<i_2\cdots<i_d}\alpha(i_1,\cdots,i_d)T_{i_1}\cdots
%T_{i_d}X_{i_1}\cdots X_{i_d}$$
% be a multi linear polynomial of degree $d$ in the variables $\{T_i X_i
%\}_{i\ge1}$. We assume that $\int_\Omega\E\left[F_n^2\right]d\P\to1$
%and $\int_\Omega\E\left[F_n^4\right]d\P\to3$ then under the
%probability $\mu\otimes\P$ we have
%$$F_n\xrightarrow[n\to\infty]{\mathrm{law}}~\mathcal{N}(0,1).$$
%The same holds, if, for instance, $\int_\Omega\E\left[F_n^4\right]d\P
%\to3$ and $\int_\Omega\E\left[F_n^8\right]d\P\to105$ as well as any
%result contained in the present article. The proof relies on the
%superposition procedure by taking into account that $F_n\in\K(\LL_
%\Omega+d\Id)$.
%\end{exm}

%s2.3 #&#
\subsection{Some auxiliary results}\label{knownfacts}

To be self-contained, we restate here two well-known facts about
Stein's method applied to eigenfunctions of a diffusive Markov
operator. For more details, the reader can consult, for instance, \cite
{le} or the survey~\cite{CP14}.

%\begin{thm}[Stein]
%Let $X,T$ be two integrable random variables such that $\E[X^2]=1$ and
%such that for any function $f\in\mathcal{C}^1_b(\R)$ we have:
%\begin{equation}\label{IPPS}
%\E\left[f'(X)T\right]=\E\left[X f (X)\right].
%\end{equation}
%Then it holds that
%\begin{equation}\label{Steinbound}
%d_{\mathrm{TV}}\left(X,\mathcal{N}(0,1)\right)\le2 \sqrt{\operatorname{Var}
%\left(T\left|\right. X\right)}.
%\end{equation}
% \end{thm}

%Now, let us translate the Malliavin Stein's method in our context. Let
%us be placed under (a)-(b)-(c). Take $X\in L^2$, with $\E[X]=0$. Then
%$X$ can be expanded in the sum of the eigenspaces of $\LL$:

%$$X=\sum_{k=1}^\infty Y_k,  Y_k\in\K(\LL+k\Id).$$
%You can define (as the serie converges in $L^2$):
%$$-\LL^{-1} X= \sum_{k=1}^\infty\frac{Y_k}{k}.$$
%Besides, the next integration by parts holds

%$$\E\left[\phi'(X)\Gamma[X,-\LL^{-1}X]\right]=\E[X\phi(X)].$$
%And one can use the bound (\ref{Steinbound}) to write
%$$d_{\mathrm{TV}}\left(X,\mathcal{N}(0,1)\right)\le2 \sqrt{\operatorname{Var}
%\left(\Gamma[X,-\LL^{-1}X]\left|\right. X\right)}\le2 \sqrt{\operatorname{Var}
%\left(\Gamma[X,-\LL^{-1}X]\right)}.
%$$
%Of course, when $X$ is an eigenfunction of $\LL$ associated with the
%eigenvalue $p$, this becomes
%\begin{equation}\label{Steindirichlet}
%d_{\mathrm{TV}}\left(X,\mathcal{N}(0,1)\right)\le2\sqrt{\operatorname{Var}\left(
%\frac{1}{p}\Gamma[X,X]\right)}
%\end{equation}
%
%th2.1 #&#
\begin{thm}[(\cite{le})]\label{Gamma-bound-general}
Let $\LL$ be a Markov diffusive operator satisfying the assumptions
\textup{(a)}--\textup{(b)} of Section~\ref{assumptions}, and $X$ be in $\K(\LL+p\Id)$
such that \mbox{$\E[X^2]=1$}. Then
\[
d_{\mathrm{TV}} \bigl(X,\mathcal{N}(0,1) \bigr)\le\frac{2}{p}\sqrt{\operatorname{Var} \bigl(\Gamma[X] \bigr)}.
\]
As a matter of fact, for a given sequence $\{X_n\}_{n \ge1}$ in $\K(\LL
+p\Id)$ such that $\E[X_n^2]\to1$:
\[
\Gamma[X_n]\stackrel{L^2}{\rightarrow}p \quad \Rightarrow \quad
X_n \stackrel {\mathit{law}} {\rightarrow}\mathcal{N}(0,1).
\]
\end{thm}

%re2.3 #&#
\begin{rmk}\label{gamma-criterion}
In \cite{le}, Proposition~2, given a sequence $\{X_n\}_{n \ge1}$ in $\K
(\LL+p\Id)$ with $\E[X_n^2]\to\theta$, it is shown that
\[
\operatorname{Var} \bigl( \Gamma[X_n] - p X_n \bigr) \to0
\quad \Rightarrow\quad  X_n+\theta\stackrel{\mathrm{law}} {\rightarrow}\gamma(
\theta),
\]
where $\gamma(\theta)$ stands for the gamma distribution of parameter
$\theta$. This fact will be used only in the proof of Theorem~\ref{unify-criteria}.
\end{rmk}

Furthermore, we restate below the fourth moment theorem under the
assumptions (a)--(b)--(c). Actually, it can be proved under the weaker
assumption that, for any eigenfunction $X\in\K(\LL+p\Id)$, we have
\[
X^2\in\bigoplus_{k\le2p} \K(\LL+k\Id),
\]
which in fact is a very particular case of the assumption (c). The
stronger assumption (c) will allow us to establish analogous
statements for higher moments.
%
%th2.2 #&#
\begin{thm}[(\cite{le,CP14})]\label{Steindirichletfinal}
Let $\LL$ be a Markov diffusive operator satisfying the assumptions
\textup{(a)}--\textup{(b)}--\textup{(c)} and $X\in\K(\LL+p\Id)$ with $\E[X^2]=1$. Then
%
%e2.15 #&#
\begin{equation}
d_{\mathrm{TV}} \bigl(X,\mathcal{N}(0,1) \bigr)\le\frac{2}{p} \sqrt{\operatorname{Var}
\bigl(\Gamma[X] \bigr)}\le\frac{2}{\sqrt{3}} \sqrt{\E\bigl[X^4
\bigr]-3}.
\end{equation}
Thus, for a given sequence $\{X_n\}_{n\ge1}$ in $\K(\LL+p\Id)$ such
that $\E[X_n^2]\to1$ and $\E[X_n^4] \to3$, we have
\[
X_n \stackrel{\mathit{law}} {\rightarrow} \mathcal{N}(0,1).
\]
\end{thm}

%s3 #&#
\section{Algebraic framework}\label{algebraic}

The aforementioned assumptions (a)--(b)--(c) on the Markov generator $\LL
$ can be suitably used to build an algebraic framework in order to
study properties of eigenfunctions of the generator $\LL$. Throughout
this section, we shall use these assumptions in a natural way in order
to introduce a family of bilinear, symmetric and positive forms
$\mathcal{M}_k$. The fundamental assumption
(\ref{fundamental-assumtion}) is the crucial element yielding the
positivity of the bilinear forms $\mathcal{M}_k$.

Let $\R_{k}[T]$ stand for the ring of all polynomials of $T$ of degree
at most $k$ over~$\R$. Let $X$ be an eigenfunction of the generator $\LL
$ with eigenvalue $-p$, that is, $- \LL X = p X$. We consider
the following map:
\[
%\label{bilinearM}
\mathcal{M}_k \dvtx  \cases{
\ds
\R_{k}[T] \times\R_{k}[T] \longrightarrow \R,
\vspace*{3pt}\cr
\ds (P,Q) \longmapsto   \E \bigl[ Q(X) (\LL+ k p \Id) P(X) \bigr].}
\]

%re3.1 #&#
\begin{rmk}
Notice that the mapping $\mathcal{M}_k$ strongly depends on the
eigenfunction $X$. We also remark that thanks to Remark~\ref{Hyperconractivity}, $\mathcal{M}_k$ is well defined.
\end{rmk}

The following theorem is the cornerstone of our approach.

%th3.1 #&#
\begin{thm}\label{thm:matrix}
The mapping $\mathcal{M}_k$ is bilinear, symmetric and nonnegative.
Moreover, its matrix representation over the canonical basis $\{ 1, T,
T^2, \break \ldots, T^k \}$ is given by $p \mathbf{M}_{k}$\vspace*{-2pt} where
%
%e3.1 #&#
\begin{equation}
\label{moment-matrix}
\mathbf{M}_{k} =
\biggl(\biggl(k - \frac{ij}{i+j-1}\biggr) \E\bigl[X^{i+j}\bigr]
\biggr)_{0 \le i,j \le k}
\end{equation}
with the convention that $\frac{ij}{i+j-1}=0$ for $(i,j)=(0,1)$ or $(1,0)$.
\end{thm}

\begin{pf}
Expectation is a linear operator, so the bilinearity property follows.
Symmetry proceeds from the symmetry of the diffusive generator $\LL$.
To prove positivity of the matrix $\mathbf{M}_k$, using the
fundamental assumption (\ref{fundamental-assumtion}) we obtain that
for any polynomial $P$ of degree $\le k$,
\[
P(X)\in\bigoplus_{i \leq k p}\K(\LL+ i \Id).
\]
Therefore, denoting by $J_i\dvtx L^2(E,\mu)\rightarrow\K ( \LL+ i \Id
 )$ the orthogonal projections,
%
%e3.2 #&#
\begin{eqnarray}
\E\bigl[ \bigl((\LL+kp\Id)P(X) \bigr)^2
\bigr]&=& \E\bigl[\LL P(X) (\LL+kp\Id)P(X)\bigr]
\nonumber\\[-2pt]
&&{}+kp \E\bigl[P(X) (\LL+kp\Id)P(X)\bigr]
\nonumber\\[-2pt]
\label{Main-ineq}
& = &\sum_{i=0}^{kp} (- i) (kp - i) \E
\bigl[J_{i}^{2}\bigl(P(X)\bigr)\bigr]
\\[-2pt]
&&{}+ kp \E\bigl[P(X) (\LL+kp\Id)P(X)\bigr]
\nonumber\\[-2pt]
\nonumber
&\le&  kp  \mathcal{M}_k(P,P).
\end{eqnarray}
Hence, $\mathcal{M}_k$ is a positive form. To complete the proof,
notice that the $(i,j)$-component of the matrix $\mathbf{M}_k$ is given
by $\E [ X^j (\LL+ kp \Id) X^i  ]$. So, using the diffusive
property of the generator $\LL$, we obtain
\begin{eqnarray*}
%\label{combination3}
\nonumber
X^j (\LL+ kp \Id) X^i
& =& i (i-1) X^{i+j-2} \Gamma(X) + p (k -i) X^{i+j}
\\
&= &\frac{i (i-1)}{i+j -1} \Gamma\bigl(X^{i+j-1}, X\bigr) + p (k -i)
X^{i+j}. %
\end{eqnarray*}
Therefore,
\begin{eqnarray*}
%\label{combination4}
%\nonumber
%
\mathcal{M}_k\bigl(X^i,X^j
\bigr)&= &\frac{i (i-1)}{i+j -1} \E\bigl[\Gamma\bigl(X^{i+j-1}, X\bigr)\bigr] + p
(k -i) \E\bigl[X^{i+j}\bigr]
\\
&=& p\frac{ i (i-1)}{i+j-1} \E \bigl[X^{i+j} \bigr] + p (k-i) \E
\bigl[X^{i+j} \bigr]
\\
&=& p \biggl( \frac{i(i-1) + (k-i)(i+j-1)}{i+j -1} \biggr) \E \bigl[X^{i+j} \bigr]
\\
&=& p \biggl(k - \frac{ij}{i+j-1} \biggr) \E \bigl[X^{i+j} \bigr].
\end{eqnarray*}
\upqed\end{pf}

%re3.2 #&#
\begin{rmk}\label{mainrmk}
In Theorem~\ref{thm:matrix}, we only stated the positivity of the
family of quadratic forms $\mathcal{M}_k$. However, it is worth
mentioning that, thanks to the inequality~(\ref{Main-ineq}), each quadratic
form $\mathcal{M}_k$ dominates the nonnegative quadratic form
\[
P \longmapsto\E\bigl[ \bigl\{(\LL+ kp \Id)P(X) \bigr\}^2\bigr].
\]
\end{rmk}

%The great advantage of encoding the spectral assumption in term of
%positiveness of a quadratic form enables us to use matrix theory to
%provide non trivial facts about the moments of eigenfunction $X$.
%As an example we have the following corollary.

%co3.1 #&#
\begin{cor}\label{determinant}
For any eigenfunction $X$ of the generator $\LL$ with eigenvalue $-p$,
that is, $- \LL(X)=pX$:
\begin{longlist}[(ii)]
\item[(i)] All the eigenvalues of matrix $\mathbf{M}_k$ are nonnegative.
\item[(ii)] All the $l$th leading principal minor of the
matrix $\mathbf{M}_k$ are nonnegative for $l\le k$.
\end{longlist}
\end{cor}

\begin{pf}
The proof follows directly from standard linear algebra (see, e.g.,
\cite{S}).
\end{pf}

The moments matrix $\mathbf{M}_k$ can help one to give nontrivial
moment inequalities, sometimes sharper than the existing estimates so
far, involving the moments of the eigenfunctions of a generator $\LL$.
Here is an application where we sharpen the standard fourth moment
inequality $\E[X^4]\geq3\E[X^2]^2$. We mention that the next theorem
unifies the two well-known criteria of convergence in law
(i.e., Gaussian and Gamma approximation) for a sequence of random
variables inside a fixed Wiener chaos; see \cite{n-p-05,n-pe-2}.

%th3.2 #&#
\begin{thm}\label{unify-criteria}
If $X$ is a nonzero eigenfunction of generator $\LL$, then
%
%e3.3 #&#
\begin{equation}
\label{4moment}
\frac{\E[X^4]}{3} - \E\bigl[X^2\bigr]^2
\ge\frac{\E[X^3]^2}{2 \E[X^2]}.
\end{equation}
Moreover, assume that $X_n \in\K(\LL+p\Id)$ for each $n \ge1$ and
%
%e3.4 #&#
\begin{equation}
\label{equalimit}
\frac{\E[X_n^4]}{3} - \E\bigl[X_n^2
\bigr]^2 -\frac{\E[X^3]^2}{2 \E[X^2]}\to0.
\end{equation}
Then all the adherence values in distribution of the sequence $\{X_n\}
_{n\ge1}$ is either a Gaussian or a scaling of a centered Gamma random
variable.
\end{thm}

\begin{pf}
The moments matrix $\mathbf{M}_2$ associated to $X$ is given by
%
%e3.5 #&#
\begin{equation}
\mathbf{M}_2(X)=
\pmatrix{ 2 & 0 & 2 \E
\bigl[X^2\bigr] \vspace*{3pt}
\cr
0 & \E\bigl[X^2\bigr] &
\E\bigl[X^3\bigr]\vspace*{3pt}
\cr
2 \E\bigl[X^2\bigr]& \E
\bigl[X^3\bigr] & \ds\tfrac{2}{3} \E\bigl[X^4\bigr] }.
\end{equation}

Using Corollary~\ref{determinant}, we infer that
\[
%\label{determi}
\operatorname{det}(\mathbf{M}_2) = 4 \E\bigl[X^2
\bigr] \biggl\{ \frac{\E[X^4]}{3} - \E \bigl[X^2\bigr]^2
\biggr\} - 2 \E\bigl[X^3\bigr]^{2} \ge0,
\]
which immediately implies (\ref{4moment}).
Up to extracting a
subsequence, we may assume that $X_n \to X_\infty$ in distribution. We
further assume that $X_\infty\neq0$. Assumption (\ref{equalimit})
entails that
\[
\det\mathbf{M}_2(X_n)\to0.
\]
Let $V_n= (\frac{2}{3}\E[X_n^4] \E[X_n^2]-\E[X_n^3]^2,   2 \E
[X_n^2] \E[X_n^3],   -2\E[X_n^2]^2)$ be the first line of the
adjugate matrix of $\mathbf{M}_2(X_n)$. Since $X_n$ converges in
distribution, we have
$V_n \to V_\infty=(a,b,c)$. We set $P(X)=c X^2+bX+a$. As a result, we have:
%
%e3.6 #&#
\begin{equation}
\label{polynomialstuff}
\mathcal{M}_2(X_n) (P,P)\to0.
\end{equation}

Using Remark~\ref{mainrmk}, we see that
\[
\E \bigl[ \bigl\{(\LL+2p\Id)P(X_n) \bigr\}^2 \bigr]
\to0.
\]

Next,
\begin{eqnarray*}
(\LL+2p\Id)P(X_n)&=&c(\LL+2p\Id)X_n^2+b p
X_n+2 a p
\\
&=& 2 c \Gamma(X_n)+b p X_n+2 a p.
%&\to0.
\end{eqnarray*}
We notice that $c\neq0$ since $X_\infty\neq0$. Now two possible
cases can happen.
\begin{longlist}[\textit{Case} (2):]
\item[\textit{Case} (1):] If $\E[X_n^3]\to0$, then $b=0$. Hence, we have $\E
 [\Gamma(X_n)+\frac{a p}{c} ]^2\to0$ and therefore the
sequence $\{X_n\}_{n \ge1}$ converges toward a Gaussian random variable.
See Theorem~\ref{Gamma-bound-general}.
\item[\textit{Case} (2):]  If $\E[X_n^3] \not\to0$, then $b\neq0$. Hence, we
have $\E[\Gamma(X_n)+\frac{b p}{2c} X_n+\frac{a p}{c}]^2\to0$. We set
$X_n=\lambda Y_n$ and we may choose $\lambda$ in such way that
\[
\operatorname{Var} \bigl(\Gamma(Y_n)-p Y_n \bigr)\to0.
\]
This enables us to use the content of the Remark~\ref{gamma-criterion}
and assert that $Y_n+\E[Y_n^2]$ converges in distribution toward a gamma
random variable. Hence, $X_n$ converges in distribution toward a
scaling of a centered gamma law.\quad\qed
\end{longlist}
\noqed\end{pf}

It is clear that Theorem~\ref{unify-criteria} also gives back the
fourth moment theorem in (a)--(b)--(c) structures from the fact that a
random variable $G$ satisfying $\E[G^2]=1$, $\E[G^3]=0$ and $\E[G^4]=3$
cannot have a gamma distribution.

The following proposition states a nontrivial inequality between the
second, fourth and sixth moments of eigenfunctions of $\LL$.

%pr3.1 #&#
\begin{prop}\label{6moment-prop}
If $X$ is an eigenfunction of $\LL$, then
%
%e3.7 #&#
\begin{equation}
\label{6moment1}
\E\bigl[X^4\bigr]^2\leq\tfrac{3}{5}
\E\bigl[X^6\bigr] \E\bigl[X^2\bigr].
\end{equation}
\end{prop}

%re3.3 #&#
\begin{rmk}
Notice that this inequality is an equality when the distribution of $X$
is\vspace*{-1pt} Gaussian.
\end{rmk}

\begin{pf*}{Proof of Proposition~\ref{6moment-prop}}
The moments matrix $\mathbf{M}_3$ associated to $X$ has the form\vspace*{-3pt}
%
%e3.8 #&#
\begin{equation}
\mathbf{M}_3=
\pmatrix{ 3 & \star& 3 \E
\bigl[X^2\bigr] & \star\vspace*{1.5pt}
\cr
\star& 2 \E
\bigl[X^2\bigr] & \star& 2 \E\bigl[X^4\bigr]\vspace*{1.5pt}
\cr
3 \E\bigl[X^2\bigr] &\star &\ds \tfrac{5}{3} \E
\bigl[X^4\bigr] & \star\vspace*{1.5pt}
\cr
\star& 2 \E
\bigl[X^4\bigr] & \star& \ds\tfrac{6}{5} \E\bigl[X^6
\bigr] }.
\end{equation}
Since this matrix is positive, we have in particular
\[
\left\llvert %
\matrix{ 2\E\bigl[X^2\bigr] & 2\E
\bigl[X^4\bigr]
\vspace*{1.5pt}\cr
2\E\bigl[X^4\bigr] & \ds\tfrac{6}{5} \E\bigl[X^6
\bigr] }
\right\rrvert \geq0,
\]
which gives the claimed inequality.\vspace*{-1pt}
\end{pf*}

Using Proposition~\ref{6moment-prop}, we can already prove the
following sixth moment theorem, that is, Theorem~\ref{CLTFMquibourre}
in the case $k=3$. Note that we will get back this result when we will
prove our main result (Section~\ref{sectioevnemoment}).\vspace*{-1pt}

%co3.2 #&#
\begin{cor}\label{6moment}
A sequence $\{X_n\}_{n \ge1}$ such that $X_n \in\K(\LL+ p \Id)$ for
each $n \ge1$, converges in distribution toward the standard Gaussian
law if and only if $\E[X_n^2]\to1$ and $ \E[X_n^6]\to15$.
\end{cor}

\begin{pf}
By Proposition  (\ref{6moment-prop}), for $X \in\K ( \LL+ p \Id
 )$, we have
\[
\E\bigl[X^6\bigr]\geq\frac{5}{3}\frac{\E[X^4]^2}{\E[X^2]}\geq
\frac{5}{3}\frac
{(3\E[X^2]^2)^2}{\E[X^2]}=15\E\bigl[X^2
\bigr]^3.
\]
Therefore, for the sequence $\{X_n\}_{n\geq1}$ in $\K ( \LL+ p
\Id )$, if $\E[X_n^2]\to1$ and $\E[X_n^6]\to15$, then from the
previous chain of inequalities, we deduce that $\E[X_n^4]\to3$.
Hence, the sequence $\{X_n\}_{n \ge1}$ converges in distribution toward
$\mathcal{N}(0,1)$ according to Theorem~\ref{Steindirichletfinal}.
\end{pf}

%s4 #&#
\section{New central limit theorems}\label{Central-limit-main}

In this section, we will establish our main criteria for central
convergence. In a first subsection, we will first focus on the main
theorem of the paper, the so-called \textit{even moment criterion}. In a
second subsection, we will give additional criteria of central
convergence. As before, we work under assumptions (a)--(b)--(c) stated in\vspace*{-2pt}
Section~\ref{setup}.
% and the Theorem~\ref{CLTFM-ultimate} will follow by choosing $\LL$
%the Ornstein-Uhlenbeck generator.~\\

%s4.1 #&#
\subsection{The even moment criterion}\label{sectioevnemoment}

We state below our main result. Note that Theorem~\ref{CLTFMquibourre}
is a particular case of Theorem~\ref{2-2k-moment}, by simply choosing
$\LL$ to be the Ornstein--Uhlenbeck generator.

%th4.1 #&#
\begin{thm}\label{2-2k-moment}
Let $\LL$ be a Markov operator satisfying \textup{(a)}--\textup{(b)}--\textup{(c)}, $p\ge1$ be an
eigenvalue of $-\LL$, and $\{X_n\}_{n\geq1}$ a sequence\vspace*{1pt} of elements in
$\K ( \LL+ p \Id )$ for all $n \ge1$, such that
$\lim_{n \to\infty}\E [X_n^2 ]=1$. Then, for any integer
$k\ge2$, as $n \to\infty$, we have
%
%e4.1 #&#
\begin{equation}
\label{2-2k-criterion} X_n \stackrel{\mathit{law}} {\rightarrow}
\mathcal{N}(0,1)\quad \mbox{if and only if} \quad\E\bigl[X_n^{2k}
\bigr] \to\E\bigl[N^{2k}\bigr]=(2k-1)!!.
\end{equation}
\end{thm}

The proof of Theorem~\ref{2-2k-moment} is rather lengthy; it is thus
divided in three steps which are detailed below.
\begin{pf*}{Sketch of the proof}
\textit{Step} (1):  We find a family $\PP=\{ W_k \vert  k \ge2\}$ of
real polynomials which satisfies the two following properties:
\begin{longlist}[(ii)]
\item[(i)]  $\E \bigl[W_k(X_n) \bigr]\ge0, \forall k
\ge2, \forall n \ge1$,
\item[(ii)]   $X_n \stackrel{\mathrm{law}} {\rightarrow}
\mathcal{N}(0,1) \mbox{ if and only if } \E \bigl[W_2(X_n)
\bigr]\to0$, as $n \to\infty$.
\end{longlist}

\textit{Step} (2):  In the second step, we construct a polynomial $T_k$
such that, under the assumptions of Theorem~\ref{2-2k-moment}, we have
\[
T_k  = \sum_{i=2}^k
\alpha_{i,k} W_i \qquad \mbox{such that for all }i,\alpha_{i,k}>0,
\]
and
\[
\E \bigl[T_k(X_n) \bigr] \to 0 \qquad  \mbox{as } n \to\infty.
\]

\textit{Step} (3):  In the last step, using the fact that $\alpha
_{i,k}>0$ and property (i) of step~(1), we obtain that $\E
[W_2(X_n) ]\to0$. Finally, using property (ii) of step (1), we
complete the proof.
\end{pf*}

\begin{pf*}{Proof of Theorem~\protect\ref{2-2k-moment}}
The ``if'' part is a
simple consequence of  Lemma~\ref{Hypercontract}. For the ``only if''
part, we go into the details of the three aforementioned steps.
\begin{longlist}[\textit{Step} (1):]
\item[\textit{Step} (1):]  First, we introduce the suitable family $\PP$ of
polynomials. To this end, we denote by $\{H_k\}_{k\ge0}$ the family of
Hermite polynomials defined by the recursive relation
%
%e4.2 #&#
\begin{equation}
H_{0}(x) =1,\qquad  H_1(x)=x,\qquad  H_{k+1}(x) = x
H_k (x)-k H_{k-1}(x).
\end{equation}
For any $k \ge2$, we define the polynomial $W_k$ as
%
%e4.3 #&#
\begin{equation}
\label{Wk}
W_k(x) = (2k-1) \biggl(  x \int_{0}^{x}
H_{k}(t)H_{k-2}(t)\, \ud t - H_{k}(x)H_{k-2}(x)
\biggr),
\end{equation}
and the family $\PP$ as
%
%e4.4 #&#
\begin{equation}
\label{eq:polynomial-family}
\PP= \Biggl\{ P  \Big\vert P(x)= \sum_{k=2}^{m}
\alpha_k W_k(x);  m \ge2,  \alpha_k \ge0, 2\le k \le m \Biggr\}.
\end{equation}

The family $\PP$ encodes interesting properties of central convergence
which are the content of the two next lemmas. Below, Lemma~\ref{main1}
will provide the answer to property (i) of step (1).
\end{longlist}
%\noqed\end{pf*}

%le4.1 #&#
\begin{lma}\label{main1}
Let $\LL$ be a general Markov generator satisfying assumptions
\textup{(a)}--\textup{(b)}--\textup{(c)} in Section~\ref{setup}, and let $P$ be a polynomial
belonging to $\PP$. Then:
\begin{longlist}[(2)]
\item[(1)] If $N\sim\mathcal{N}(0,1)$, $\E[P(N)]=0$.
\item[(2)] If $X$ is an eigenvalue of $\LL$, $\E[P(X)]\geq0$.
\end{longlist}
\end{lma}
\begin{pf}
It is enough to prove that $E[W_k(X)]\geq0$ and $E[W_k(N)]=0$. Using
the diffusive property (\ref{diff}), the fact that $-\LL X = p X$ and
the recursive property of Hermite polynomials, we obtain that
%
%e4.5 #&#
\begin{eqnarray}
 (\LL+ k p  \Id) H_{k}(X)& =&  H_{k}^{\prime\prime}(X)
\Gamma(X) + H_{k}^{\prime}(X) \LL(X) + k p H_{k}(X)
\nonumber\\
& = & H_{k}^{\prime\prime}(X) \Gamma(X) - p X H_{k}^{\prime}(X)
+k p H_{k}(X)
\nonumber
\\[-8pt]
\label{computation1}
\\[-8pt]
\nonumber
&= & H_{k}^{\prime\prime}(X) \bigl( \Gamma(X) - p \bigr)
\\
\nonumber
&=&  k(k-1) H_{k-2}(X) \bigl( \Gamma(X) - p \bigr).
\end{eqnarray}
Therefore,
%
%e4.6 #&#
\begin{eqnarray}
\mathcal{M}_k(H_k)&=& \E \bigl[
H_{k}(X) ( \LL+ k p  \Id ) H_{k}(X) \bigr]
\nonumber
\\[-8pt]
\label{combination22}
\\[-8pt]
\nonumber
&= & k (k-1) \E \bigl[ H_{k}(X) H_{k-2}(X) \bigl( \Gamma(X) -
p \bigr) \bigr].
\end{eqnarray}
Next, by the integration by parts formula (\ref{by-parts}), we have
%
%e4.7 #&#
\begin{eqnarray}
&& \E \bigl[ H_{k}(X) H_{k-2}(X)
\bigl( \Gamma(X) - p \bigr) \bigr]
\nonumber\\
&&\qquad = \E \biggl[ \Gamma \biggl(\int
_{0}^{X} H_{k}(t) H_{k-2}(t)
\,\ud t , X \biggr) \biggr]
- p \E \bigl[ H_{k}(X) H_{k-2}(X) \bigr]
\nonumber
\\[-8pt]
\label{combination2}
\\[-8pt]
\nonumber
&&\qquad=p \E \biggl[ X \int_{0}^{X}
H_{k}(t) H_{k-2}(t)\, \ud t - H_{k}(X)
H_{k-2}(X) \biggr]
\\
\nonumber
&&\qquad= \frac{p}{2k-1}  \E \bigl[ W_k(X) \bigr].
\end{eqnarray}
Hence,
\[
\mathcal{M}_k(H_k)=\frac{pk(k-1)}{2k-1}  \E \bigl[
W_k(X) \bigr],
\]
and the inequality $\E[W_k(X)]\geq0$ follows from the positivity of
the bilinear form $\mathcal{M}_k$. Finally, choosing $\LL$ to be the
Ornstein--Uhlenbeck generator and $X=N$ a standard Gaussian random
variable living in the first Wiener chaos (i.e., $p=1$) with variance
$1$, then $\Gamma(N)=p=1$ and computation (\ref{combination2}) shows
that $\E[W_k(N)]=0$ for every $k\geq2$. Hence,
$\E[P(N)]=0$ for every $P \in\PP$.
\end{pf}

The next lemma is central in the proof of the even moment Theorem~\ref
{2-2k-moment}. In fact, the next lemma will provide answer to property
(ii) of step (1).

%le4.2 #&#
\begin{lma}\label{carac}
Assume that $\LL$ be a general Markov generator satisfying assumptions
\textup{(a)}--\textup{(b)}--\textup{(c)} of Section~\ref{setup}. Let $p \ge1$ and $\{X_n\}_{n\geq
1}$ a sequence\vspace*{1pt} of elements in $\K ( \LL+ p \Id )$
for all $n \ge1$. Let $P=\sum_{k=2}^m \alpha_k W_k \in\PP$ such that
$\alpha_2\neq0$. Then, as $n \to\infty$, we have
\[
X_n \stackrel{\mathit{law}} {\rightarrow} \mathcal{N}(0,1)\quad \mbox{if
and only if}\quad \E\bigl[P(X_n)\bigr] \to\E\bigl[P(N)\bigr]=0.
\]
\end{lma}

\begin{pf}
In virtue of Lemma~\ref{main1},
\begin{eqnarray*}
\E\bigl[P(X_n)\bigr]&=&\sum_{k=2}^m
\alpha_k \E\bigl[W_k(X_n)\bigr]
\\
&\ge& \alpha_2 \E\bigl[W_2(X_n)\bigr]
\\
&=&\alpha_2 \bigl(\E\bigl[X_n^4\bigr]-6\E
\bigl[X_n^2\bigr]+3 \bigr).
\end{eqnarray*}
This leads to
\[
0\leq\E\bigl[X_n^4\bigr]-6\E\bigl[X_n^2
\bigr]+3 \leq\frac{1}{\alpha_2} \E\bigl[P(X_n)\bigr].
\]
By assumption, $\E[P(X_n)]\to0$, so $\E[X_n^4]-6\E[X_n^2]+3 \to0$. On
the other hand,
\[
\E\bigl[X_n^4\bigr]-6\E\bigl[X_n^2
\bigr]+3=\E\bigl[X_n^4\bigr]-3 \E\bigl[X_n^2
\bigr]^2+3\bigl(\E\bigl[X_n^2\bigr]-1
\bigr)^2.
\]
Thus, we obtain that $\E[X_n^2] \to1$ and $\E[X_n^4]\to3$, and we can
use Theorem~\ref{Steindirichletfinal} to conclude.
\end{pf}
\begin{longlist}[\textit{Step} (2):]
\item[\textit{Step} (2):] This step consists in finding a suitable polynomial
$T_k \in\PP$ of the form
%
%e4.8 #&#
\begin{equation}
\label{Tk}
T_k (x) = x^{2k} - \alpha_k
x^2 + \beta_k,\qquad  \alpha_k, \beta_k
\in \R.
\end{equation}
\end{longlist}

To find such a polynomial, notice that according to step (1), the
function $\phi_k\dvtx x\mapsto\E[T_k(xN)]$ must be positive and vanish at
$x=1$. Hence, we must have
$\phi_k(1)=\phi_k '(1)=0$. This leads us to the following system of equations:
\[
\cases{
(2k-1)!!-\alpha_k+
\beta_k=0,
\vspace*{2pt}\cr
2k (2k-1)!! -2\alpha_k=0.}
\]
Therefore, the coefficients $\alpha_k$ and $\beta_k$ are necessarily
given by
\[
\alpha_k = k (2k-1)!! \quad\mbox{and}\quad \beta_k= (k-1)
(2k-1)!!.
\]
It remains to check that the corresponding polynomial $T_k(x)=x^{2k} -
k  (2k-1)!!  x^2 + (k-1) (2k-1)!! \in\PP$. To this end, one needs
to show that $T_k$ can be expanded over the basis
$\{W_k\}_{k \geq2}$ with positive coefficients. We answer to this by
the affirmative with the next proposition, which also provides an
explicit formula for the coefficients.

%pr4.1 #&#
\begin{prop}\label{tenduduslip}
Let $k \geq2$, and $T_k(x)=x^{2k} - k  (2k-1)!!  x^2+ (k-1)
(2k-1)!!$. Then
%
%e4.9 #&#
\begin{equation}
\label{eq:expx2kk}
T_k(x)=\sum_{i=2}^k
\alpha_{i,k} W_i(x),
\end{equation}
where
\[
\alpha_{i,k} = \frac{(2k-1)!!}{2^{i-1}(2i-1)(i-2)!}
\pmatrix{k
\cr
i}
\int_0^1(1-u)^{-1/2}
u^{i-2} \biggl(1-\frac{u}{2} \biggr)^{k-i} \,du.
\]
In particular, $T_k\in\PP$ and $\alpha_{2,k}>0$ for all $k\ge1$.
\end{prop}

The proof of this proposition is rather involved and can be found in
the \hyperref[app]{Appendix}.
\begin{longlist}[\textit{Step} (3):]
\item[\textit{Step} (3):]  Let $p \ge1$. Assume that $\{X_n\}_{n\geq1}$ is a
sequence\vspace*{1pt} of elements of $\K(\LL+p\Id)$ for all $n \geq1$ such that
$\lim_{n\to\infty} \E [X_n^2 ]=1$. We further assume that
$\E [X_n^{2k} ]\to(2k-1)!!$. Using step (2), we have
\begin{eqnarray*}
\E \bigl[T_k(X_n) \bigr]&=&\E \bigl[X_n^{2k}
\bigr]-k (2k-1)!! \E \bigl[X_n^2\bigr]+(k-1) (2k-1)!!
\\
%&=&\E\left[X_n^{2k}\right]-(2k-1)!!\\
&\to& 0.
\end{eqnarray*}
To finish the proof, by step (2), we know that $T_k\in\PP$ and
$c_{2,k}>0$. Thus, Lemma~\ref{carac} applies and one gets the desired
conclusion.\quad\qed
\end{longlist}
\noqed\end{pf*}

We end this section with the following result containing a quantitative
version of the Theorem~\ref{2-2k-moment}. We remark that item (1) of
Theorem~\ref{last-thm} contains Theorem~\ref{superquantification} in
the \hyperref[sec1]{Introduction} by assuming $\LL$ to be Ornstein--Uhlenbeck operator.
%
%th4.2 #&#
\begin{thm}\label{last-thm}
Let $\LL$ be a Markov operator satisfying assumptions \textup{(a)}--\textup{(b)}--\textup{(c)} of
Section~\ref{setup}. Let $p\ge1$ and $X$ be an eigenfunction of $\LL$
with eigenvalue $p$ such that $\E[X^2]=1$. Assume that $k \ge2$. Then
\begin{longlist}[(2)]
\item[(1)] We have the following general quantitative bound:
%
%e4.10 #&#
\begin{equation}
\label{boundclt}
d_{\mathrm{TV}} \bigl(X,\mathcal{N}(0,1) \bigr)\le
C_k\sqrt{\frac{\E
[X^{2k} ]}{(2k-1)!!}-1},
\end{equation}
where the constant $C_k=\frac{4}{\sqrt{ 2k(k-1) \int_0^1 (({1+t^2})/{2})^{k-2} \,dt}}$.
\item[(2)] The moment estimate $\E[X^{2k}] \ge\E[N^{2k}]=(2k-1)!!$ holds.
\end{longlist}
\end{thm}

\begin{pf}
Taking into account Remark~\ref{mainrmk}, for any polynomial $P=\break \sum_{k=2}^m \alpha_k W_k$ in family $\PP$, we obtain that
\[
\E\bigl[P(X)\bigr]\geq\frac{1}{p^2}\sum_{k=2}^m
(2k-1) (k-1)\alpha_k\E \bigl[ H_{k-2} (X)^2
\bigl( \Gamma(X) - p \bigr)^2 \bigr].
\]
By applying the latter bound to $P=T_k$ and using Proposition~\ref
{tenduduslip}, we infer that
\begin{eqnarray*}
\E\bigl[T_k(X)\bigr]&\geq& \frac{1}{p^2}\sum
_{i=2}^m (2i-1) (i-1)\alpha_{i,k}\E
\bigl[ H_{i-2} (X)^2 \bigl( \Gamma(X) - p
\bigr)^2 \bigr]
\\
&\ge& \frac{3 \alpha_{2,k}}{p^2} \E \bigl[ \bigl( \Gamma(X) - p \bigr)^2
\bigr].
\end{eqnarray*}
On the other hand, Proposition~\ref{tenduduslip} shows that
\[
\alpha_{2,k}=\frac{(2k-1)!!}{6}\pmatrix{k
\cr
2}\int
_0^1 (1-u)^{-{1}/{2}} \biggl(1-
\frac{u}{2}\biggr)^{k-2}\,du.
\]
This leads us to
%
%e4.11 #&#
\begin{eqnarray}
 \qquad\E\bigl[X_n^{2k}\bigr]-(2k-1)!!
&\geq & \biggl(
\frac{(2k-1)!! }{4}k (k-1) \int_0 ^1
\frac{1}{\sqrt{1-u}} \biggl(1-\frac{u} 2 \biggr)^{k-2}\,du \biggr)
\nonumber
\\[-8pt]
\\[-8pt]
\nonumber
&&{}\times
\E \biggl[ \biggl(\frac{\Gamma(X_n)}{p}-1 \biggr)^2 \biggr].
\end{eqnarray}
Now,\vspace*{-1pt} the desired inequality follows from 
Theorem~\ref{Steindirichletfinal} and identity\break $\int_0 ^1 \frac{1}{\sqrt
{1-u}}(1-\frac{u} 2)^{k-2}\,du=2 \int_0^1 (\frac{1+t^2}{2})^{k-2} \,dt$. We
stress that with taking
$k=2$ in \eqref{boundclt}, we recover the well-known bound (see,
e.g., \cite{Optimal,n-p-2}):
\[
d_{\mathrm{TV}} \bigl(X_n,\mathcal{N}(0,1) \bigr)\le
\frac{2}{\sqrt
{3}}\sqrt{\E \bigl[X_n^4
\bigr]-3}.
\]
The\vspace*{1pt} second item (2) easily follows from the fact that $\E[T_k(X)] \ge
0$. When $\E[X^2]\neq1$, using the normalized random variable $\tilde
{X}=\frac{X}{\sqrt{\E[X^2]}}$, we obtain the inequality
$\E[X^{2k}] \ge\E^k[X^2] \E[N^{2k}]$ for all $k \ge1$.
\end{pf}

%re4.1 #&#
\begin{rmk}
The statement (2) of Theorem~\ref{last-thm} does not hold for any kind
of Markov operators. Below, we present a simple counterexample.\vspace*{1pt} Let $U$
denote a uniform random variable on the interval $(-1,1)$. Set
$X=U^2- \frac{1}{3}$. Then\vspace*{1pt} $X$ belongs to the second Wiener chaos of
the Jacobi structure (see \cite{a-c-p}, Section~4) with parameters
$\alpha=\beta=1$. Besides, $\E[X^2]=\frac{4}{45}$. Then it is
straightforward to check that the inequality $\E[X^{2k}] \ge\E[N^{2k}]
\E^k[X^2]$ in the item (2) of Theorem~\ref{last-thm} does not hold even
for $k=2$.
This is mainly because the assumption (c) fails in this setup. Roughly
speaking, the spectrum of Jacobi operators has a quadratic growth
whereas our assumption suggests a linear growth.
\end{rmk}

%re4.2 #&#
\begin{rmk}\label{studentexample}
Here, we give a concrete application of Theorem~\ref{last-thm} in some
situation where the usual criteria in the Wiener space fail. Let $\nu
\ge1$ be an integer number. Assume that $\{Q_n\}_{n \ge1}$ is a
sequence of i.i.d. random variables having chi-squared distribution
with $\nu$ degrees of freedom. We are also given $\{N_n\}_{n\ge1}$ an
independent sequence of i.i.d. standard Gaussian random variables. As a result,
$ \{ S_n \}_{n \ge1}= \{ N_n \times\sqrt{\frac{\nu
}{Q_n}}  \}_{n \ge1}$ is a sequence of i.i.d. Student random
variables with $\nu$ degrees of freedom. Now, set
\[
X=\sum_{1= i_1<i_2<\cdots<i_p}^{\infty} \alpha(i_1,
\ldots,i_p) S_{i_1}\cdots S_{i_p},
\]
such that $\E [X^2 ]=1$. Relying on the superposition
procedure (see Section~\ref{examples-assumptions}) and Theorem~\ref
{last-thm}, if $\nu> 2k$, it can be shown that
%
%e4.12 #&#
\begin{equation}
\label{student-ex}
d_{\mathrm{TV}} \bigl(X,\mathcal{N}(0,1) \bigr)\le
C_k \sqrt{\frac{\E
[X^{2k} ]}{(2k-1)!!}-1}.
\end{equation}
In addition, since $X$ does not have moments of all orders, $X$ does
not belong to any Wiener chaos and therefore the estimate (\ref
{student-ex}) is strictly beyond existing moments-based total-variation
estimates on Wiener space.
\end{rmk}

\subsection{Other polynomial criteria for central convergence}

In the previous section, in order to prove the even moment theorem, we
use heavily the fourth moment Theorem~\ref{Steindirichletfinal}. The
reason is that in the decomposition of $T_k$ over the
basis $\{W_k\}_{k\geq2}$, the coefficient $\alpha_2$ in front of $W_2$
is strictly positive. It is then natural to consider the cases where
$\alpha_2 =0$, which turns out to be more delicate. The main result
of this section is the following.

%th4.3 #&#
\begin{thm}\label{thm:main1}
Let $\LL$ be a general Markov generator satisfying assumptions
\textup{(a)}--\textup{(b)}--\textup{(c)} in Section~\ref{setup}. Assume that $\{X_n\}_{n \ge1}$ is
a sequence of eigenfunctions of~$\LL$ with eigenvalue $-p$, that is, $-
\LL X_n = p X_n$ for
each $n$. We suppose that $P=\sum_{k=2}^m \alpha_k W_k$ is a nonzero
polynomial belonging to the family $\PP$, such that as $n \to\infty$,
we have
%
%e4.13 #&#
\begin{equation}
\label{poly-condi}
\E\bigl[P(X_n)\bigr] \to\E\bigl[P(N)\bigr]=0.
\end{equation}

Then, as $n \to\infty$, the two following statements hold:
\begin{longlist}[(2)]
\item[(1)] If there exist at least two indices $2 <i < j$ such that
$\alpha_i \alpha_j>0$ and $i$ or $j$ is even, then
\[
X_n \stackrel{\mathit{law}} {\rightarrow}\mathcal{N}(0,1).
\]
\item[(2)] If there exist at least two indices $2< i < j$ such that
$\alpha_i \alpha_j>0$ and both $i$ and $j$ are odd integers, then each
accumulation point of sequence $\{X_n\}_{n \ge1}$ in distribution is
in the form
\[
\alpha\mathcal{N}(0,1)+(1-\alpha)\delta_0
\]
for some $\alpha\in[0,1]$.
\end{longlist}
\end{thm}

\begin{pf}
We will consider each case separately.
\begin{longlist}[\textit{Case} (1):]
\item[\textit{Case} (1):]  Let us notice that there exist $A>0$ and $B\in\R$
such that $\forall x\in\R,  P(x)\ge Ax^2+B$. Then $Ax^2<P(x)-B$. By
assumption, $\E[P(X_n)]\rightarrow0$, so $\E[P(X_n)-B]$ is bounded and
$\E[X_n^2]$ is bounded as
well. Hence, by Lemma~\ref{Hypercontract}, the sequence $\{X_n\}_{n\ge
1}$ is bounded in $L^p(E,\mu)$ for each $p\ge1$. Since $\Gamma
(X_n)=\frac{1}{2}(\LL+2p\Id)[X_n^2]$, and because of the fact that $\LL
$ is a continuous operator when its domain is restricted to a finite
sum of eigenspaces of $\LL$, $\Gamma(X_n)$ is also bounded in any
$L^p(E,\mu)$. Finally, up to extracting a subsequence, we may assume
that the sequence of random vectors $ \{(X_n,\Gamma(X_n)) \}
_{n\ge1}$ converges in distribution toward a random vector $(U,V)$. As
a consequence of Remark~\ref{mainrmk}, we have
\begin{eqnarray*}
\E\bigl[H_{i-2}(X_n)^2 \bigl(
\Gamma[X_n]-p \bigr)^2\bigr] &\to & 0,
\\
\E\bigl[H_{j-2}(X_n)^2 \bigl(
\Gamma[X_n]-p \bigr)^2\bigr] &\to& 0.
\end{eqnarray*}

Recalling that $\{(X_n,\Gamma(X_n))\}_{n\ge1}$ converges in
distribution toward $(U,V)$, we infer that almost surely
%
%e4.14 #&#
\begin{equation}
\label{intermediarystep}
H_{i-2}(U) (V-p )=H_{j-2}(U) (V-p )=0.
\end{equation}
Thus, on the set $\{V\neq p\}$, we have $H_{i-2}(U)=H_{j-2}(U)=0$. But
the roots of two Hermite polynomials of different orders are distinct
if at least one of the orders is even. By assumption, either
$i-2$ or $j-2$ is even, and we conclude that $\P (V\neq p )=0$.
This proves that any accumulation point (in distribution) of the
sequence $\{\Gamma(X_n)\}_{n\ge1}$ is $p$, and, as a consequence, the
sequence $\Gamma(X_n)$ converges to $p$ in~$L^2$. Now, we can conclude
by using Theorem~\ref{Gamma-bound-general}.

\item[\textit{Case} (2):]  Following the same line of reasoning as in case (1),
we obtain:
\[
H_{i-2}(U) (V-p )=H_{j-2}(U) (V-p )=0,\qquad  \mbox{a.s.}
\]
On the set $\{V\neq p\}$, we have $H_{i-2}(U)=H_{j-2}(U)=0$. But the
roots of two Hermite polynomials with odd orders only coincide at $0$.
This implies $U (V-p)=0$ almost surely. Now, let $\phi$ be any test
function. Using the integration by parts formula \eqref{by-parts} with
$Y=X_n$ and $X=\phi(X_n)$ and letting $n \rightarrow+\infty$, one
leads to
%
%e4.15 #&#
\begin{equation}
\label{intermediary2}
\E\bigl[\phi'(U) V\bigr]=p\E\bigl(U\phi(U)\bigr].
\end{equation}
Splitting the expectations in \eqref{intermediary2} into the disjoint
sets $\{V=p\}$ and $\{V\neq p\}$, we obtain
%
%e4.16 #&#
\begin{equation}
\label{intermediary3}
p\E\bigl[ \bigl(\phi'(U)-U\phi(U) \bigr)
\mathbh{1}_{\{V=p\}}\bigr]+\phi'(0)\E[V\mathbh{1}_{\{V\neq p\}}]=0.
\end{equation}
Take $\phi(x)=e^{i\xi x}$. Then \eqref{intermediary3} reads
\[
p i\xi\E\bigl[e^{i\xi U}\mathbh{1}_{\{V=p\}}\bigr]-p\E\bigl[U
e^{i\xi U}\mathbh{1}_{\{
V=p\}}\bigr]+i \xi\E[V\mathbh{1}_{\{V\neq p\}}]=0.
\]
Setting $f(\xi)=\E[e^{i\xi U}\mathbh{1}_{\{V=p\}}]$, we obtain that
\begin{eqnarray*}
&\ds p \xi f(\xi)+p f'(\xi)+\xi\E[V\mathbh{1}_{\{V\neq p\}}]=0,&
\\
&\ds f(\xi)= \biggl(\P(V=p)- \frac{1}{p} \E[V] \biggr)+ \frac{1}{p}
\E[V] e^{-{\xi^2}/{2}}.&
\end{eqnarray*}
It is straightforward to deduce from above equations that the
characteristic function of random variable $U$ is given by
\begin{eqnarray*}
\E\bigl[e^{i\xi U}\bigr]&=&\P(V\neq p)+f(\xi)
\\
%&=&\P(V\neq p)+\Big( \P(V=p) - \frac{1}{p} \E[V] \Big) + \frac{1}{p}
%\E[V] e^{-\frac{\xi^2}{2}}\\
&=& \biggl(1- \frac{1}{p} \E[V] \biggr) +
\frac{1}{p} \E[V] e^{-{\xi^2}/{2}}.
\end{eqnarray*}
\end{longlist}
%\upqed
\upqed\end{pf}

Although case (2) in Theorem~\ref{thm:main1} seems less interesting
than case (1), we point out that a Dirac mass at zero may appear
naturally under assumptions (a)--(b)--(c). Here is a simple example of this
phenomenon.

%ex4.1 #&#
\begin{exm}
Set $E=\R^2$ and $\mu=\mathcal{N}(0,1)\otimes (\frac{1}{2}\delta
_0+\frac{1}{2}\delta_1 )$. Define
%
%e4.17 #&#
\begin{equation}
\label{newgene}
\LL[\phi](x,y)=y \biggl(\frac{\partial^2\phi}{\partial x ^2} - x\frac
{\partial\phi}{\partial x}
\biggr).
\end{equation}
One can check that $\LL$ fulfills assumptions (a)--(b)--(c) in Section~\ref{setup}. Consider the sequence
\[
X_n(x,y)=x y \in\K(\LL+\Id),\qquad n\ge1.
\]
Then $X_n\sim\frac{1}{2}\mathcal{N}(0,1)+\frac{1}{2}\delta_0$ for each
$n\ge1$. Moreover, $\E [W_3(X_n) ]=\break \E [W_5(X_n)
]=0$. As a matter of fact, the conclusions of Theorem~\ref{thm:main1}
are sharp when applied to $P=W_3+W_5 \in\PP$.
\end{exm}

However, we show that in the particular setting of the Wiener space,
that is, when $\LL$ is the Ornstein--Uhlenbeck operator, the case (2)
of Theorem~\ref{thm:main1} cannot take place. Furthermore, condition
(\ref{poly-condi}) will be a necessary and sufficient condition for
central convergence. To this end, we need the following lemma, which
has an interest on its own.

%le4.3 #&#
\begin{lma}\label{product-lemma}
Let $\{ U_n\}_{n \ge1}$ and $\{V_n\}_{n \ge1}$ be two bounded
sequences such that for some integer $M>0$, we have
\[
U_n, V_n \in \bigoplus_{i=0}^M
\K(\LL+ i \Id)\qquad \forall n\in\N.
\]
If $\E [ U_n^2 V_n^2 ] \to0$ as $n$ tends to infinity, then
$\E[U_n^2]\E[V_n^2] \to0$ as $n$ tends to infinity.
\end{lma}

We will make use of the next theorem, due to Carbery--Wright, restated
here for convenience. More precisely, we will apply it to Gaussian
distribution, which is log-concave.
%We emphasize that this is the main ingredient of our proof because of
%the fact that the Normal and Gamma distributions are log-concave, so
%that one can use the Carbery-Wright inequality.

%th4.4 #&#
\begin{thm}[(\cite{c-w}, Carbery--Wright)]\label{cw-thm}
Assume that $\mu$ is a log-concave probability measure on $\R^m$. Then
there exists an absolute constant $c>0$ (\textrm{independent of $m$ and
$\mu$}) such that for any polynomial $Q\dvtx \R^m\to\R$ of degree at most
$k$ and any
$\alpha>0$, the following estimate holds:
%
%e4.18 #&#
\begin{equation}
\label{cw-ineq}
\biggl(\int Q^2\,d\mu \biggr)^{{1}/({2k})}\times\mu
\bigl\{ x\in\R^m\dvtx  \bigl|Q(x)\bigr|\leq\alpha \bigr\} \leq c k
\alpha^{{1}/k}.
\end{equation}
\end{thm}

\begin{pf*}{Proof of Lemma~\ref{product-lemma}}
Let us denote $E=\R^\N, \mu=\mathcal{N}(0,1)^{\otimes\N}$ and let $\LL
$ be the Ornstein--Uhlenbeck generator. We assume that $\E [U_n^2
]$ does not converge to zero. Up to extracting a subsequence, we can
suppose that
$\E [U_n^2 ]>\theta>0$ for each $n\ge1$. Following the method of
\cite{n-po-2}, page~659, inequality (3.21), we can approximate in
$L^2(E,\mu)$ the random variable $U_n$ by polynomials of degree $M$.
Hence, applying the Carbery--Wright inequality for the approximating
sequence, and taking the limit, we obtain
%
%e4.19 #&#
\begin{equation}
\label{ineqCW}
\mu\bigl\{x\in E\dvtx  \bigl|U_n(x)\bigr|\leq\alpha\bigr\}\leq
\frac{c M \alpha^{1/M}}{\theta
^{1/2M}}\leq K \alpha^{1/M},
\end{equation}
with $K=\frac{cM}{\theta^{1/2M}}$. Next, we have the following inequalities:
\begin{eqnarray*}
\E \bigl[V_n^2 \bigr]&=&\E \biggl[V_n^2
\frac{U_n^2}{U_n^2} \textbf{1}_{\bigl\{
|U_n|>\alpha\bigr\}} \biggr]+\E \bigl[V_n^2
\textbf{1}_{\bigl\{|U_n|\leq\alpha\bigr\}} \bigr]
\\
&\leq&\frac{1}{\alpha^2}\E \bigl[U_n^2
V_n^2 \bigr]+\sqrt{\E \bigl[V_n^4
\bigr]}\sqrt{\mu\bigl\{x\in E\dvtx  \bigl|U_n(x)\bigr|\leq\alpha\bigr\}}
\\
&\le&\frac{1}{\alpha^2}\E \bigl[U_n^2
V_n^2 \bigr]+C K \alpha^{1/2M},
\end{eqnarray*}
where $K$ is the constant from the Carbery--Wright inequality and $C$
is such that $\sup_{n\ge1}\E [V_n^4 ]\leq C^2$. Note\vspace*{1pt} that
constant $C$ exists by hypercontractivity
(see Remark~\ref{Hyperconractivity}). We immediately deduce that
\[
\limsup_{n\to\infty}\E \bigl[V_n^2 \bigr]
\leq C K \alpha^{1/2M},
\]
which is valid for any $\alpha>0$. Let $\alpha\to0$ to achieve the proof.
\end{pf*}

%th4.5 #&#
\begin{thm}\label{main}
Let $\LL$ stand for the Ornstein--Uhlenbeck operator and let $\{X_n\}
_{n\geq1}$ be a sequence of elements of $\K(\LL+p\Id)$ with variance
bounded from below by some positive constant. Then, for any nonzero
polynomial $P\in\PP$, as $n \to\infty$, we have
\[
X_n \stackrel{\mathit{law}} {\rightarrow} \mathcal{N}(0,1)\quad \mbox{if
and only if}\quad \E\bigl[P(X_n)\bigr]\to0.
\]
\end{thm}
%
%\begin{rmk}
%More particularly, if $P(0)\neq0$ then $\E\left[P(X_n)\right]\to0~
%\Rightarrow X_n\xrightarrow[n\to\infty]{\mbox{Law}}~\mathcal{N}(0,1).$
%\end{rmk}

\begin{pf}
Although in Theorem~\ref{main} we assume that $\LL$ is the
Ornstein--Uhlenbeck generator, we stress that the proof works in the
Laguerre structure or any tensor products of Laguerre and Wiener structures.
The ``if'' part is straightforward by using the continuous mapping
theorem. To show the ``only if'' part, we take a nonzero polynomial
$P\in\PP$ of the form
\[
P(x)=\sum_{k=2}^m \alpha_k
W_{k}(x),
\]
with $\alpha_m >0$. Thanks to Remark~\ref{mainrmk}, as $n \to\infty$,
we know that
%
%e4.20 #&#
\begin{equation}
\E \bigl[H_{m-2}(X_n)^2\bigl(
\Gamma(X_n)-p\bigr)^2 \bigr] \rightarrow0.
\end{equation}
Let $Z_{m-2}=\{ t_1, t_2, \ldots, t_{m-2}\}$ be the set of the (real)
roots of the Hermite polynomial $H_{m-2}$. Then, as $n \to\infty$, we have
\[
\E \Biggl[ \Biggl(\prod_{k=1}^{m-2}
(X_n - t_k)^2 \Biggr) \bigl(\Gamma
(X_n)-p\bigr)^2 \Biggr] \to0.
\]

From the fact that $\Gamma(X_n)=\frac{1}{2}(\LL+p\Id)(X_n^2)$ together
with fundamental assumption (\ref{fundamental-assumtion}) (which
holds in the Wiener structure), we deduce that $H_{m-2}(X_n)$ and
$\Gamma(X_n)-p$ are both finitely expanded over the eigenspaces of the
generator $\LL$. Besides, repeating the same argument as in the proof
of Theorem~\ref{thm:main1}, we can show that the sequence
$\{X_n\}_{n\ge1}$ is bounded in $L^2(E,\mu)$, as well as $\{ \Gamma
(X_n)-p\}_{n\ge1}$. Thus, from Lemma~\ref{product-lemma}, as $n \to
\infty$, we obtain
\[
\Biggl(\prod_{k=1}^{m-2} \E
\bigl[(X_n - t_k)^2 \bigr] \Biggr) \E \bigl[
\bigl(\Gamma (X_n)-p\bigr)^2 \bigr] \to0.
\]

Since $\E [(X_n - t_k)^2 ]\geq\operatorname{Var}(X_n)$ is bounded from
below by assumption, we conclude that $\Gamma(X_n)\to p$ in $L^2(E)$.
Hence, using Theorem~\ref{Gamma-bound-general}, we obtain that the
sequence $\{X_n\}_{n \ge1}$ converges in distribution toward $\mathcal{N}(0,1)$.
\end{pf}

%s5 #&#
\section{Conjectures}

The main motivation of this article is to provide an answer to the
question (B) stated in the \hyperref[sec1]{Introduction}. We have shown that the
convergence of any even moment guarantees the central convergence of a
normalized sequence (i.e., $\E[X_n^2] \to1$) living inside $\K(\LL+p
\Id)$. In the latter criterion, we have dealt with normalized sequences
because it seems more natural from the probabilistic
point of view. However, one could also try to replace this assumption
by the convergence of another even moment. Indeed, our framework could
provide a wider class of polynomial conditions ensuring central
convergence, namely through the family $\PP$. Then it is natural to
check whether the family $\PP$ is rich enough to produce other pair of
even moments ensuring a criterion for central convergence. To be more
precise, assume that for some pair $(k,l)$ ($k <l$) of positive
integers, we have $\E[X_n^{2k}]\to\E[N^{2k}]$ and $\E[X_n^{2l}]\to\E
[N^{2l}]$, we want to know if this implies a central convergence. Our
method would consist in deducing the
existence of a nontrivial polynomial $T_{k,l} \in\PP$ such that
$\E [ T_{k,l}(X_n)  ] \rightarrow0$. Natural candidates are
polynomials of the form
\[
T_{k,l}(x) = x^{2l} + \alpha x^{2k} + \beta,
\]
where $\alpha, \beta\in\R$. Using the same arguments as in the step
(2) of the proof of Theorem~\ref{2-2k-moment}, one can show that the
condition $P \in\PP$ entails necessarily that $\alpha= \frac{l
(2l-1)!!}{k (2k-1)!!}$
and $\beta=  ( \frac{l} k -1  ) (2k-1)!!$. Then the question
becomes: does the polynomial $T_{k,l}$ belong to family $\PP$?

We exhibit the decomposition of $T_{k,l}$ for each pair of integers in
the set $\Theta= \{(2,3); (2,4); (2,5); (3,4); (3,5) \}$:
\begin{eqnarray*}
 T_{2,3}(x)&=&x^6 -
\frac{15}{2}x^4 +\frac{15}{2} = W_{3}(x) +
\mathbf{\frac{5}{2}}W_2(x),
\\
T_{2,4}(x)&=&x^8 -70 x^4 +105 =
W_4(x) + \mathbf{\frac
{84}{5}}W_3(x)+\mathbf{28}
W_2(x),
\\
T_{2,5}(x)&=&x^{10} - \frac{1575}{2}x^4 +
\frac{2835}{2} \\
&=& W_5(x) + \mathbf{\frac{180}{7}}
W_4(x) +\mathbf{234} W_3(x) + \mathbf{\frac{585}{2}}
W_2(x),
\\
T_{3,4}(x)&=&x^8 -\frac{28}{3} x^6 +35
= W_4(x) + \mathbf{\frac
{112}{5}} W_3(x)+ \mathbf{
\frac{14}{3}} W_2(x),
\\
T_{3,5}(x)&=&x^{10}-105x^6 +630 =
W_5(x) + \mathbf{\frac{180}{7}}W_4(x) +
\mathbf{129}W_3(x) + \mathbf{30}W_2(x).
\end{eqnarray*}
The coefficients of each decomposition are positive, thus, for each
pair $(k,l) \in\Theta$, the convergence of the $2k{\mathrm{th}}$ and
$2l{\mathrm{th}}$ moments entails the central convergence. Naturally, we are
tempted to formulate the following conjecture.

%co1 #&#
\begin{con}
Let $k,l \geq2$ be two different positive integers. For any sequence
$\{X_n\}_{n \ge1}$ of eigenfunctions in the same eigenspace of a
Markov generator $\LL$ satisfying assumptions (a)--(b)--(c), as
$n \to\infty$, the following statements are equivalent:
\begin{longlist}[(i)]
\item[(i)] $X_n \stackrel{\mathit{law}}{\longrightarrow} N \sim\mathcal{N}(0,1)$.\vspace*{1pt}
\item[(ii)] $\E[X_n^{2k}] \to\E[N^{2k}]$ and $\E[X_n^{2l}] \to\E[N^{2l}]$.
\end{longlist}
\end{con}

Unfortunately, we could not prove it since $T_{4,5}$ does not belong to
family $\PP$:
\[
T_{4,5}(x)=x^{10} -\frac{45}{4}x^8 +
\frac{945}{4}= W_5(x) + \frac
{405}{28}W_4(x) +
W_3(x)\,\mathbf{-}\, \mathbf{\frac{45}{2}}W_2(x).
\]
We insist on the fact that the above conjecture might be true
nonetheless.

Another perspective of our algebraic framework is to provide nontrivial
moments inequalities for the eigenfunctions of the Markov operator $\LL
$ satisfying suitable assumptions. The special role of the
fourth cumulant $\kappa_4$ in normal approximation for a sequence
living inside a fixed eigenspace is now well understood and it is known
that $\kappa_4(X) \ge0$. In a recent preprint, the authors of
\cite{app14} observed the prominent role of $\kappa_6$ for studying
convergence in distribution toward $N_1 \times N_2$, where $N_1$ and
$N_2$ are two
independent $\mathcal{N}(0,1)$ random variables, of a given sequence in
a fixed Wiener chaos. The computations suggest that $\kappa_6$ could be
greater than the variance of some
differential operator (analogous to $\operatorname{Var} (\Gamma
[X,X ] )$ in the case of normal approximation). However, the
techniques presented in \cite{app14} could not provide the positivity
of the sixth cumulant. We recall that
\[
\label{kappa6} \kappa_6(X)= \E\bigl[X^6\bigr] -15\E
\bigl[X^2\bigr]\E\bigl[X^4\bigr]-10\E\bigl[X^3
\bigr]^2+30\E\bigl[X^2\bigr]^3.
\]
Computations show that the least eigenvalue of the moment matrix
$\textbf{M}_3(X)$ is always bigger than $\kappa_6(X)$. Therefore, our
method does not give results precise enough, to insure the positivity of
the sixth cumulant. However, we know that $\kappa_6(X) \ge0$ in the
two first Wiener chaoses. Moreover, using Proposition~\ref{6moment-prop}, we could prove the following partial criterion.

%pr5.1 #&#
\begin{prop}
Let $X$ be a multiple Wiener--It\^o integral of odd order such that $\E
[X^2] =1$. If $\kappa_4(X) \ge3$, then $\kappa_6(X)\ge0$.
\end{prop}

These two facts lead us to formulate the following conjecture.

%co2 #&#
\begin{con}
For any multiple Wiener--It\^o integral $X$ of order $p \ge2$, we
have $\kappa_6(X)>0$.
\end{con}

%sA #&#
\begin{appendix}
\section*{Appendix}\label{app}

We give here a proof of Proposition~\ref{tenduduslip}. In the
following, $w$ stands for the density of the standard Gaussian
distribution over $\R$. Let us begin by stating a~lemma on elementary
computations on Hermite polynomials.

%leA.1 #&#
\begin{lma}
\label{lem:hkhkp2}
Let $l,m,n \in\mathbb N$. Then
%
%eA.1 #&#
\begin{equation}
\label{eq:lem1} \int_\R x^{2m}
H_{2n}(x) w(x) \,dx =\frac{(2m)!}{2^{m-n}(m-n)!}
\end{equation}
and
%
%eA.2 #&#
\begin{eqnarray}
&& \int_\R H_l(x)
H_m(x) H_n(x) w(x) \,dx
\nonumber
\\[-8pt]
\label{eq:lem2}
\\[-8pt]
\nonumber
&& \qquad= \frac{l!   m!   n!}{ (
({-l+m+n})/{2}  )! ( ({l-m+n})/{2}  )! (({l+m-n})/{2})!},
\end{eqnarray}
with the convention that $\frac{1}{p!} = 0$ if $p \notin\N$.
\end{lma}

\begin{pf}
We first focus on \eqref{eq:lem1}. Recall that $e^{-{x^2}/{2}}
H_n(x) = (-1)^n\times\break   \frac{d^n}{dx^n} ( e^{-{x^2}/{2}})$.
Performing $2n$ integrations by parts (with $n\le m$), we obtain
\begin{eqnarray*}
\int_\R x^{2m} H_{2n}(x) w(x) \,dx &=&
\int_\R\frac{d^{2n}}{dx^{2n}} \bigl(x^{2m}\bigr) w(x)
\,dx
\\
&=&\frac{(2m)!}{(2(m-n))!} \int_\R x^{2(m-n)} w(x) \,dx
\\
&=&\frac{(2m)! (2(m-n)-1)!!}{(2(m-n))!}
\\
&=&\frac{(2m)!}{2^{m-n} (m-n)!}.
\end{eqnarray*}
If $m>n$, the formula follows from our convention. Now, \eqref{eq:lem2}
is a mere consequence of the product formula for Hermite polynomials,
which states that (see, e.g., Theorem~6.8.1 in
\cite{andrews1999special})
\[
H_n(x) H_m(x) = \sum_{k=0}^{\mathrm{min} (n,m)}
\pmatrix{n
\cr
k }
\pmatrix{
m
\cr
k }
 k! H_{n+m-2k}(x),
\]
for all positive integers $n,m$. Indeed, integrating last equation
against $H_l w$, and using the orthogonality of Hermite polynomials
with respect to $w$, we obtain the desired result.
\end{pf}

Now, let us prove Proposition~\ref{tenduduslip}.
\begin{pf*}{Proof of Proposition~\protect\ref{tenduduslip}}
To make the notation less cluttered, we set $\beta_k = (k-1)
(2k-1)!!$ and $\alpha_k = k   (2k-1)!!$. Since $W_p$ is an even
polynomial and $\deg(W_p) = 2p$, there exists a unique expansion of the form
%
%eA.3 #&#
\begin{equation}
\label{eq:expx2k}
x^{2k}-\alpha_k x^2 +
\beta_k=\sum_{p=2}^k
c_{p,k} W_p(x) + ax^2+b.
\end{equation}
Recall that the coefficients $\alpha_k$ and $\beta_k$ are chosen in
such a way that $\phi(t)=\E [t^{2k} N^{2k}-\alpha_k t^2 N^2+\beta
_k ]$ satisfies $\phi(1)=\phi'(1)=0$. Coming back to Lemma~\ref
{main1}, for each $p\ge2$ the two following conditions hold:
\[
\cases{
\ds \E \bigl[W_p(N) \bigr]=0,\vspace*{2pt}\cr
\ds\forall x\in\R,\qquad \psi_p(x)=\E \bigl[W_p(x N) \bigr]\ge0.}
\]
Thus, $\psi_p$ reaches its minimum at $x=1$ and we have $\psi_p(1)=\psi
_p'(1)=0$. Setting
\[
\psi(x)=\E \Biggl[\sum_{p=2}^k
c_{p,k}W_p(x N) \Biggr]=\sum_{p=2}^k
c_{p,k} \psi_p(x),
\]
we must also have $\psi(1)=\psi'(1)=0$. Plugging\vspace*{1pt} the above conditions
on $\phi$ and $\psi$ into \eqref{eq:expx2k} implies that, if $\delta
(x)=\E [ax^2 N^2+b ]=ax^2+b$, then $\delta(1)=\delta'(1)=0$.
Hence, $a+b=0$ and $2a=0$ so $a=b=0$. Define the (even) polynomial
$Q_k(x) = \sum_{p=2}^k c_{p,k} (2p-1)   H_{p}(x) H_{p-2}(x)$. Using
the definition of $W_p$ and \eqref{eq:expx2k}, we see that $Q_k$ is
solution of the polynomial equation
%
%eA.4 #&#
\begin{equation}
\label{polynomialequation}
x\int_0^x Q_k(t)
\,dt - Q_k(x) = x^{2k}-\alpha_k x^2 +
\beta_k.
\end{equation}
In the following lemma, we solve the above equation.
%\end{pf*}

%leA.2 #&#
\begin{lma}\label{solupoly}
Equation \eqref{polynomialequation} has a unique even polynomial
solution of degree $2k-2$, which is
%
%eA.5 #&#
\begin{equation}
\label{eq:decompohkhkm2}\qquad Q_k(x) = \sum_{p=2}^k
c_{p,k} (2p-1) H_{p}(x) H_{p-2}(x) = -\beta
_k+\sum_{p=1}^{k-1}
\frac{(2k-1)!!}{(2p-1)!!} x^{2p}.
\end{equation}
\end{lma}
\begin{pf}
%\textit{(Lemma~\ref{solupoly})}
Let $\Phi$ be the linear
operator from $\R[X]$ to $\R[X]$ defined by $\Phi(P) (X)=X \int_0^X
P(t)\,dt-P(X)$. Assume that $\Phi(P)=\Phi(Q)$, then $\Delta(X)=\int_0^X
 (P(t)-Q(t) ) \,dt$ satisfies the differential equation $x y (x)
-y'(x)=0$. Thus, there exists $C>0$ such that $\Delta(X)=C e^{{X^2}/{2}}$. But $\Delta$ is a polynomial function so $C=0$. This
implies that $P-Q$ is a constant polynomial. By setting $x=0$ in
equation \eqref{polynomialequation}, we get that $Q_k(0)=-\beta_k$.
Now, set
\[
R_k(X)=-\beta_k+\sum_{p=1}^{k-1}
\frac{(2k-1)!!}{(2p-1)!!} X^{2p},
\]
we also have $R_k(0)=-\beta_k$. As a result, one is left to show that
$\Phi(R_k)=\Phi(Q_k)$. Indeed,
\begin{eqnarray*}
\Phi(R_k)&=&-\beta_k \bigl(X^2-1\bigr)+\sum
_{p=1}^{k-1} \frac{(2k-1)!!}{(2p-1)!!} \biggl(
\frac{1}{2p+1} X^{2p+2}-X^{2p} \biggr)
\\
&=&-\beta_k \bigl(X^2-1\bigr)+(2k-1)!!\sum
_{p=1}^{k-1}\frac{1}{(2p+1)!!} X^{2p+2}
\\
&&{}-(2k-1)!!\sum_{p=1}^{k-1}
\frac{1}{(2p-1)!!}X^{2p}
\\
&=&-\beta_k\bigl(X^2-1\bigr)+ X^{2k}-(2k-1)!!
X^2
\\
&=& X^{2k}-\alpha_k X^2+\beta_k
\\
&=&\Phi(Q_k).
\end{eqnarray*}
\upqed\end{pf}

Integrating \eqref{eq:decompohkhkm2} against $H_{2n}  w$ over $\R$ for
each $1\leq n \leq k-1$ and using Lemma~\ref{lem:hkhkp2} shows that $\{
c_{p,k} \}_{2\leq p \leq k}$ is the solution of the following
triangular array:\vspace*{-2pt}
\begin{eqnarray*}
&&\sum_{p=n+1}^{k} c_{p,k} (2p-1)
\frac
{(p-2)!   p!   (2n)!}{(n+1)!   (n-1)!   (p-n-1)!}
\\[-2pt]
&&\qquad =\sum_{p=n+1}^{k} \frac{(2k-1)!!  (2p-2)!}{
(2p-3)!!2^{p-n-1}(p-n-1)!} \qquad \forall n \in[1,k-1],
\end{eqnarray*}
which can be equivalently stated as
%
%eA.6 #&#
\begin{equation}
\label{eq:eqnarray}
\qquad\forall n \in[1,k-1],  \qquad  \sum_{p=n}^{k-1}
\frac{a_{p,k}}{(p-n)!}= \frac{ 2^{n}(n+1)!(n-1)!}{(2n)!} \sum_{p=n}^{k-1}
\frac{ p!}{ (p-n)!},
\end{equation}
by denoting, for all $1\leq p \leq k-1$,
%
%eA.7 #&#
\begin{equation}
\label{eq:apcp}
a_{p,k} = \frac{(2p+1)(p-1)! (p+1)!}{(2k-1)!!} c_{p+1,k}.
\end{equation}

%Step 3.
In order to solve \eqref{eq:eqnarray}, we introduce the polynomial functions
\[
f(x) = -k+ \sum_{p=0}^{k-1}
x^{p},
\qquad
g(x) = \sum_{p=1}^{k-1} \frac{a_{p,k} }{p!}
x^p.
\]
Remark that, in terms of the functions $f$ and $g$, \eqref{eq:eqnarray}
reads
\[
\forall n \in[1,k-1], \qquad  g^{(n)}(1)= \frac{
2^{n}(n+1)!(n-1)!}{(2n)!}
f^{(n)}(1).
\]
The multiplication formula for the Gamma function and a classic
property of the beta function (see, e.g., \cite{abramowitz1972handbook}, formulas (6.1.20) and (6.2.2)) imply
\begin{eqnarray*}
\frac{ 2^{n}(n+1)!(n-1)!}{(2n)!} &=& 2^n
\frac{\Gamma(n+2) \Gamma
(n)}{\Gamma(2n+1)}
\\
& =& \frac{\Gamma(n+2)}{2^n\Gamma(n+1)} \cdot \frac{\Gamma(1/2)\Gamma
(n)}{\Gamma(n+1/2)}
\\
&= & \frac{n+1}{2^n} \int_0^1
u^{n-1}(1-u)^{-1/2}\,du.
\end{eqnarray*}
Thus, $\forall x \in(1/4,3/4)$,
\begin{eqnarray*}
g(1-2x) -g(1)&= & \sum_{n=1}^{k-1}
\frac{g^{(n)}(1)}{n!} (-1)^n 2^n x^n
\\
&= & \sum_{n=1}^{k-1} \frac{f^{(n)}(1)}{n!}
(n+1)\int_0^1 u^{n-1}(1-u)^{-1/2}\,du
(-1)^n x^n
\\
&=&  \int_0^1(1-u)^{-1/2} \sum
_{n=1}^{k-1} \frac{f^{(n)}(1)}{n!} (n+1)
(-1)^n u^{n-1}x^n \,du
\\
&= & \int_0^1(1-u)^{-1/2}
u^{-1}\frac{d}{du} \bigl( u f(1-ux) \bigr) \,du.
\end{eqnarray*}
Since $f(x) = -k + \frac{1-x^k}{1-x}$,
\[
\frac{d}{du } \bigl[ u f(1-ux) \bigr] =\frac{d}{du} \biggl[ -k u+
\frac{1-(1-ux)^{k}}{x} \biggr]= k \bigl( (1-ux)^{k-1} -1 \bigr),
\]
so that, $\forall x\in(1/4,3/4)$,
\[
g(1-2x) -g(1) = k\int_0^1(1-u)^{-1/2}
u^{-1} \bigl( (1-ux)^{k-1} -1 \bigr) \,du.
\]
Derive last equation to obtain that $\forall p \in[1, k-1]$, $\forall
x \in(1/4,3/4)$,
\[
2^p g^{(p)}(1-2x) = \frac{k!}{(k-1-p)!}\int
_0^1(1-u)^{-1/2} u^{p-1}
(1-ux)^{k-1-p} \,du.
\]
Note that we used Lebesgue's derivation theorem, which applies since
\[
\sup_{x \in(1/4,3/4)} \bigl| (1-u)^{-1/2} u^{p-1}
(1-ux)^{k-p-1} \bigr| \leq(1-u)^{-1/2} u^{p-1} \biggl(1-
\frac{u}{4} \biggr)^{k-p-1},
\]
and the upper bound in the last equation is in $L^1((0,1))$ as a
function of $u$.
Finally, for all $1\leq p \leq k-1$,
\[
a_{p,k} = g^{(p)}(0) = 2^{-p} \frac{k!}{(k-p-1)!}
\int_0^1(1-u)^{-1/2} u^{p-1}
\biggl(1-\frac{u}{2} \biggr)^{k-p-1} \,du,
\]
and we can use \eqref{eq:apcp} to conclude.
\end{pf*}
\end{appendix}

\section*{Acknowledgments}
The authors thank Giovanni Peccati and Lauri Viitasaari for many useful
discussions. We are grateful to two anonymous referees for their
valuable comments that led to an improved version of the previous work.

% imsref loaded by daiva.urboniene, 2015-01-20 10:38:24

%\begin{appendix}
%\section{}
%

%\section*{Acknowledgments}

%\begin{supplement}[id=suppA]
%\slink[doi]{10.1214/00-AOPXXXXSUPP} %[doi,text={...}] - jei reikia
%suskaldyti doi
%\sfilename{aopXXXX\_supp.pdf}
%\end{supplement}

%\begin{thebibliography}{99}
%\bibitem[\protect\citeauthoryear{}{}]{r1}
%\bibitem{r1}
%\end{thebibliography}

\printaddresses
\end{document}